\documentclass[11pt,a4paper,reqno]{amsart}
\usepackage{version}
\usepackage{amsmath, amssymb,amsfonts,amsthm}
\usepackage{esint}
\usepackage{fancyhdr}
\usepackage{color}   
\usepackage{hyperref}
\hypersetup{
    colorlinks=true, 
    linktoc=all,     
    linkcolor=blue,  
    citecolor=red,
}

\usepackage{xcolor}

\usepackage{amsmath,amsthm,amssymb,latexsym,epsfig,graphicx,subfigure}
\usepackage{color}
\usepackage{parskip}
\usepackage{hyperref}
\numberwithin{equation}{section}

\newtheorem{theorem}{Theorem}[section]

\newtheorem{lemma}{Lemma}[section]
\newtheorem{example}[theorem]{Example}

\newtheorem{conjecture}{Conjecture}

\newtheorem{corollary}{Corollary}[section]
 \newtheorem{proposition}[theorem]{Proposition}
\theoremstyle{definition}

\theoremstyle{definition}
\newtheorem{remark}{Remark}[section]

\newcommand{\R}{\mathbb{R}}

\newcommand{\F}{\mathbb{F}}

\newcommand{\e}{\varepsilon}

\newcommand{\spt}{\textrm{spt }}

\newtheorem{innercustomthm}{Theorem}
\newenvironment{customthm}[1]
  {\renewcommand\theinnercustomthm{#1}\innercustomthm}
  {\endinnercustomthm}
\allowdisplaybreaks

\hypersetup{backref=true}

\begin{document}
\title{Packing sets in Euclidean space by affine transformations}
\author{Alex Iosevich}
\address{Department of Mathematics, University of Rochester, USA.} \email{alex.iosevich@rochester.edu} 
\thanks{A.I. was supported in part by the National Science Foundation Grant NSF DMS - 2154232.}
\author{Pertti Mattila}
\address{Department of Mathematics and Statistics, University of Helsinki, Finland.}
\email{pertti.mattila@helsinki.fi}
\author{Eyvindur Palsson}
\address{Department of Mathematics, Virginia Tech, USA.}
\email{palsson@vt.edu}
\author{Minh-Quy Pham}
\address{Department of Mathematics, University of Rochester, USA.} \email{qpham3@ur.rochester.edu}
\author{Thang Pham}
\address{University of Science, Vietnam National University, Hanoi, Vietnam.}
\email{thangpham.math@vnu.edu.vn}
\author{Steven Senger}
\address{Department of Mathematics, Missouri State University, USA.}
\email{stevensenger@missouristate.edu}
\author{Chun-Yen Shen}
\address{Department of Mathematics, National Taiwan University, Taiwan.}
\email{cyshen@math.ntu.edu.tw}

\begin{abstract} For Borel subsets $\Theta\subset O(d)\times \R^d$ (the set of all rigid motions) and $E\subset \R^d$, we define
\begin{align*}
    \Theta(E):=\bigcup_{(g,z)\in \Theta}(gE+z).
\end{align*}
In this paper, we investigate the Lebesgue measure and Hausdorff dimension of $\Theta(E)$ given the dimensions of the Borel sets $E$ and $\Theta$, when $\Theta$ has product form. We also study this question by replacing rigid motions with the class of dilations and translations; and similarity transformations. The dimensional thresholds are sharp. Our results are variants of some previously known results in
the literature when $E$ is restricted to smooth objects such as spheres, $k$-planes, and surfaces.


\end{abstract} 
\maketitle
\tableofcontents

\section{Introduction}\label{sec_1}
\subsection{Packing problem in Euclidean space}
Let $\mathcal{A}$ be a metric space of maps $f:\R^d\to \R^d$, $d\geq 1$.

For Borel sets $E\subset \R^d$, $\Lambda\subset \mathcal{A}$, we define
\begin{align*}
    \Lambda(E):=\bigcup_{f\in \Lambda}f(E)=\{f(x): x\in E, f\in \Lambda\}\subset \R^d.
\end{align*}
The packing problem in this setting can be formulated as follows: Is it possible for a set of zero $d-$dimensional Lebesgue measure to contain an image $f(E)$ of $E$ for every $f\in \Lambda$?

The study of the packing problem has a long history, see for example \cite{Falconer85book}, \cite{Mattila15}, and \cite{Keleti17}. For the planar case, lines and circles are special cases of curve-packing problems.
If we consider $E$ as a line segment in the plane, and let $\mathcal{A}$ be the set of all rigid motions, then the above problem dates back to the works of Besicovitch \cite{Besicovitch20}, \cite{Besicovitch28} in the 1920s. In these papers, he constructed a set named after him which has zero Lebesgue measure and contains a line segment of unit length in every direction. Now let $E=S^1$, the unit circle, let $\mathcal{A}=\R^2\times \R_{>0}$, and define for each $\gamma=(z,r)\in \mathcal{A}$, $x\in \R^2$,
\begin{align*}
    \gamma(x)=rx+z.
\end{align*}
Then $\Lambda(S^1)$ is the union of all circles centered at $z$, of radius $r$, for $(z,r)\in \Lambda$. In 1968, Besicovitch and Rado \cite{BesicovitchRado68}, and Kinney \cite{Kinney68} showed that there exists a set in the plane of Lebesgue measure zero containing a circle of radius $r$ for each $r>0$ (see also Davies \cite{Davies72}). Later, in 1980, Talagrand \cite{Talagrand80} proved that there exists a set of measure zero containing a circle centered at $x$, for all $x$ on a given straight line.
Thus it is natural to ask which conditions will guarantee that the union of lines (circles) has positive measure.

In this paper, we will study the following question, which generalizes the question for lines and circles above: Under which conditions on the dimensions of $E$ and $\Lambda$, the Lebesgue measure of $\Lambda(E)$ is positive? Note that when this holds, we certainly can not pack $\Lambda(E)$ into any set of zero Lebesgue measures. Similarly, we can ask: Given a constant $0<u\leq d$, how large do the Hausdorff dimensions $\dim_\mathcal{H} E$ and $\dim_{\mathcal{H}}\Lambda$ need to be to ensure that
\begin{align*}
    \dim_{\mathcal{H}} \Lambda(E)\geq u
\end{align*}
holds? 

For circles in the plane, in 1985, Bourgain \cite{Bourgain85}, and independently Marstrand \cite{Marstrand87}, demonstrated that if the centers form a set of positive Lebesgue measure, then the union of the circles must have positive Lebesgue measure, which answer a question by Falconer \cite{Falconer85book} (see also Bourgain \cite{Bourgain86}). Earlier Stein \cite{Stein76}, as a consequence of his work on the spherical means maximal operator, proved that the same conclusion holds true for spheres in $\R^d, d\geq 3$. The case $d=2$ turned out to be much more difficult.
Since the spheres have dimension $d-1$, we can expect that it is enough for the centers of the spheres to form a set of Hausdorff dimension bigger than one to give a positive result. This was shown by Mitsis \cite{Mitsis99} for $d\geq 3$ and by Wolff \cite{Wolff00} for $d=2$ (see also D. Oberlin \cite{Oberlin06}). In a paper published in 1997, Wolff \cite{Wolff97} also showed that if the set of centers has Hausdorff dimension $s$, $0<s\leq 1$, the corresponding union of circles has dimension of at least $1+s$. For results in higher dimensions, see D. Oberlin \cite{Oberlin07}. 

We can get similar results by replacing the circle in the plane with rather general smooth curves in $\R^d$.   
If $E$ is a nondegenerate (i.e., derivatives span the whole space) curve in $\R^d$, $d\geq 3$, Ham, Ko, Lee, and Oh \cite{Hametal23} showed that if the set of translations and dilations has dimension $>\alpha$, for $0<\alpha\leq d-1$, then the union of all curves has dimension $\geq \alpha+1$. For results in the plane, see \cite{Kaenmaki17}. Simon and Taylor
\cite{SimonTaylor20}, \cite{SimonTaylor22} studied properties of $A+\Gamma$, where $\Gamma$ is a planar curve with at least one point of non-vanishing curvature. For general results on the packing of curves, surfaces, and manifolds, see for example Falconer \cite{Falconer82}, Wisewell \cite{Wisewell04}.  

Turning to the case of affine hyperplanes in $\R^d$, we denote the Grassmanian manifold of all affine hyperplanes by $\mathcal{G}(d,d-1)$. 
Then if the set of hyperplanes has Hausdorff dimension larger than $1$, the union of these hyperplanes has a positive Lebesgue measure. In fact, we can say more, if $0<s<1$, and the dimension of the set of hyperplanes is $s$,
then the union of hyperplanes will have dimension at least $d-1+s$. These results are due to D. Oberlin \cite{Oberlin07}. Later, in \cite{Oberlin14}, he generalized to affine $k$-planes in $\R^d$. More precisely, he showed that if the set of affine $k$-planes has Hausdorff dimension $\alpha> (k+1)(d-k)-k$, then the union of corresponding $k$-planes has a positive Lebesgue measure. The result is sharp, in the sense that for every $\e>0$, there exists a set of $k$-planes of dimension $(k+1)(d-k)-k-\e$ such that the union of these $k-$planes has zero Lebesgue measure. In 2016, R. Oberlin \cite{ROberlin16} conjectured that if a set of lines $L$ in $\R^d$ has dimension $\geq 2(k-1)+\beta$, with integer $1\leq k\leq d$, then the union of these lines has dimension at least $k+\beta$. Recently, this conjecture was proved by Zahl \cite{Zahl23}. For the variant of this problem when $L$ is replaced by a set of $k-$dimensional affine subspaces in $\R^d$, see \cite{FalconerMattila16}, \cite{Heraetal19}, \cite{Hera19}, and \cite{Gan23}. For results on packing skeletons of polytopes, see \cite{Keletietal18}, \cite{Thorton17}, and \cite{Changetal18}.

In this paper, instead of limiting the study to certain smooth objects, we will consider the packing problem for general Borel sets in $\R^d$.  In particular, we are interested in packing problems in the case where $\mathcal{A}$ is the set of dilations and translations; rigid transformations; or similarity transformations.

Similar questions for rigid motions in finite field settings have also been studied in \cite{Pham23}.

\subsection{Packing sets by dilations and translations}\label{subsec_dil}
In this section, we will discuss the results of packing problems by using dilations and translations. 

Let $d\geq 1$, we denote $\Gamma(d)=\R^d\times \R_{>0}$ as the set of all translations and dilations. For each $\gamma=(z,r)\in \Gamma(d)$ and $x\in \R^d$, we define $\gamma(x)=rx+z$. For $E\subset \R^d$, and $\Gamma\subset \Gamma(d)$, set
\begin{align*}
    \Gamma(E)=\bigcup_{\gamma=(z,r)\in \Gamma}(rE+z)=\{\gamma(x):\gamma\in \Gamma, x\in E\}.
\end{align*}
In other words, $\Gamma(E)$ is the union of all dilated and translated copies $\gamma(E)$ of $E$, with $\gamma\in \Gamma$.
If $\Gamma$ is of the form $Z\times T$, where $Z\subset \R^d$, $T\subset \R_{>0}$, by denoting $TE=\{ rx: x\in E, r\in T\}$, we can write
\begin{align*}
    \Gamma(E)= TE+Z.
\end{align*}
We restate the packing problem using dilations and translations as follows: For Borel sets $E$ and $\Gamma$, can a set of zero $d-$dimensional Lebesgure measure contains dilated and translated copies $\gamma(E)=rE+z$ of $E$ for each element $\gamma=(z,r)\in \Gamma$?

Recently, by assuming that the set of dilations and translations has product form $Z\times T\subset \R^d\times \R_{>0}$, 
Hambrook and Taylor \cite{HambrookTaylor21}
proved the following theorem. Throughout this paper, $\dim_{\mathcal F}$ denotes the Fourier dimension, see section \ref{sec_2} for details.

\begin{customthm}{A}\cite[Theorem 1.1]{HambrookTaylor21}\label{thm_Hambrook_Taylor}
    Let $T\subset \R_{>0}$, $E,Z\subset \R^d$, be non-empty compact sets.   Let
    $        \delta=  \max\{ \dim_{\mathcal{F}}(TE)+\dim_{\mathcal{H}} Z, \dim_{\mathcal{H}}(TE)+\dim_{\mathcal{F}} Z\}.$
    \begin{itemize}
        \item[(i)] If $
        \delta>d$, then $\mathcal{L}^d(TE+Z)>0$.
        \item[(ii)]  If $
        \delta \leq d$, then  $\dim_{\mathcal{H}}(TE+Z)\geq \delta$.
    \end{itemize}
\end{customthm}

\begin{remark}\label{remark_1}
        Theorem \ref{thm_Hambrook_Taylor} is a variant of some classical results for smooth objects. For example, when $E$ is the unit sphere $S^{d-1}$ in $\R^d$, this theorem recovers results of Wolff \cite{Wolff97}, \cite{Wolff00} and D. Oberlin \cite{Oberlin06} in the case the set of centers and radii has product structure, which we mentioned in the introduction. However, Theorem \ref{thm_Hambrook_Taylor} is still not a complete generalization, since the much deeper results    
        by Wolff and D. Oberlin take the union of spheres over an arbitrary set $\Gamma$ of translations and dilations, while Theorem \ref{thm_Hambrook_Taylor} takes the union of spheres over the Cartesian product set $\Gamma=Z\times T$ (see also the discussion in \cite{HambrookTaylor21}).
\end{remark}
     Theorem \ref{thm_Hambrook_Taylor} holds in the more general form as follows. 
    Throughout this paper, $\dim_{\mathcal S}$ denotes the Sobolev dimension. We say that a set $A\subset \R^d$ has Sobolev dimension $\dim_\mathcal{S}A\geq s$ if $A$ carries a Borel probability measure $\mu$ such that
\begin{align*}
    \int |\widehat{\mu}(x)|^2(1+|x|)^{s-d}dx<\infty.
\end{align*}
     
     \begin{customthm}{B}\label{thm_B}
        Let  $A,B\subset \R^d$ be Borel sets. Then
         \begin{itemize}
         \item[(i)] $\dim_\mathcal{S}(A+B)\geq\dim_{\mathcal{F}}A+\dim_{\mathcal{H}}B.$
             \item[(ii)] If $\dim_{\mathcal{F}}A+\dim_{\mathcal{H}}B>d$, then $\mathcal{L}^d(A+B)>0.$
         \item[(iii)] If for $0<u<d$, $\dim_{\mathcal{F}}A+\dim_{\mathcal{H}}B>u$, then $\dim_\mathcal{H}(A+B)\geq u.$
         \end{itemize}
     \end{customthm}
     The proof of Theorem \ref{thm_B} is as in \cite{HambrookTaylor21} (see Section \ref{sec_2} for details).

While Theorem \ref{thm_Hambrook_Taylor} is sharp (for example, take $E=\{0\}\subset \R^d$, $T=\{1\}$ and $Z\subset \R^d$ such that $\dim_HZ=d$, while $\mathcal{L}^d(Z)=0$), the role of the Hausdorff dimension of $E$ in this problem is not provided. Additionally, finding lower bounds for $\dim_\mathcal{F}(TE)$ in Theorem \ref{thm_Hambrook_Taylor} is important.

Our first main result is the following, which answers the above questions.
\begin{theorem}\label{thm_main_dilation_translation}
    Let $d\geq 1$. Let $E,Z\subset \R^d$, $T\subset \R_{>0}$ be Borel sets, and put $\Gamma=Z\times T$. Then we have the following:
     \begin{itemize}
     \item[(i)] $\dim_\mathcal{F}TE\geq\dim_\mathcal{H}E+\dim_\mathcal{H}T-d$, and 
     $$\dim_{\mathcal{S}} \Gamma(E)\geq \dim_\mathcal{H}E+\dim_\mathcal{H}Z+\dim_\mathcal{H}T-d.$$
        \item[(ii)] If 
        $\dim_\mathcal{H}E+\dim_\mathcal{H}Z+\dim_\mathcal{H}T>2d$,
               then $$\mathcal{L}^d(\Gamma(E))>0.$$
    \item[(iii)] For $0<u<d$, if
    $\dim_\mathcal{H}E+\dim_\mathcal{H}Z+\dim_\mathcal{H}T>d+u$, 
     then $$\dim_{\mathcal{H}} \Gamma(E)\geq u.$$
      \end{itemize}
\end{theorem}

\begin{remark}\label{remark_1.2}
    \begin{itemize}
        \item[]
        \item[(i)] The bounds in Theorem \ref{thm_main_dilation_translation} $(i)$, and $(ii)$ are sharp.
        Indeed, for part $(ii)$, one can take $E=[0,1]^d$, $T=\{0\}$, $Z\subset \R^d$ such that $\dim_\mathcal{H}Z=d$, and $\mathcal{L}^d(Z)=0$. Then $\dim_\mathcal{H}E+\dim_\mathcal{H}Z+\dim_\mathcal{H}T=2d$, while $\mathcal{L}^d(TE+Z)=\mathcal{L}^d(Z)=0$. The sharpness of part $(i)$ follows by the sharpness of part $(i)$. For other examples, see \ref{Section_examples}.
        \item[(ii)] When $d=1$, Theorem \ref{thm_main_dilation_translation} is the result of Bourgain \cite[Theorem 7]{Bourgain10}.  Also, when $d=1$, part $(ii)$ follows from Falconer's exceptional estimate for projections \cite{Falconer82} and part $(iii)$ from a recent result of Ren and Wang \cite{RenWang23}.
    \end{itemize}
\end{remark}
To emphasize the importance of the estimates for lower bounds of the Fourier dimension, we give an immediate corollary of Theorem \ref{thm_main_dilation_translation} to the $k$-fold sum-product set without proof, see \cite[Proposition 3.14]{Mattila15}. Let $A\subset \R^d$, we denote $A^{(k)}=\underbrace{A+\cdots+A}_{k}$ as the $k-$fold sum-set of $A$. 

\begin{corollary}\label{cor1.1}
    Let $d\geq 1$, $k\geq 1$. Let $E\subset \R^d$ and $T\subset \R_{>0}$ be Borel sets.
    \begin{itemize}
        \item[(i)] If $\dim_{\mathcal{H}}E+\dim_{\mathcal{H}}T>d\left(1+\frac{1}{k}\right)$, then
        \begin{align*}
            \mathcal{L}^d\left((TE)^{(k)}\right)=\mathcal{L}^d(TE+\dots+TE)>0.
        \end{align*}
        \item[(ii)] For $0<u<d$, if $\dim_{\mathcal{H}}E+\dim_{\mathcal{H}}T> d+\frac{u}{k}$, then
        \begin{align*}
    \dim_\mathcal{H}\left((TE)^{(k)}\right)\geq u.
    \end{align*}
        \item[(iii)] If $\dim_{\mathcal{H}}E+\dim_{\mathcal{H}}T>d\left(1+\frac{2}{k}\right)$, then $(TE)^k$ has non empty interior.
    \end{itemize}
\end{corollary}

\begin{remark}
\begin{itemize}
    \item[]
    \item[(i)] For $d\geq 1$, we can find set $E\subset \R^d$ such that $\dim_\mathcal{H}E=d$, $\mathcal{L}^d(E)=0$, while $E^{(k)}=E+\cdots+E$ has $\mathcal{L}^d(E^{(k)})=0$, for all $k\geq 1$, see Section \ref{Section_examples} for details.
    \item[(ii)] When $d=1$, $(i),(ii)$ follow from a result of Erdo\u gan, Hart, and Iosevich \cite{Erdoganetal13}. 
\end{itemize}
\end{remark}
We also obtain a similar result when $\Gamma$ is a general set in $\Gamma(d)$. 

\begin{theorem}\label{thm_main_dilation_translation_general}
    Let $d\geq 1$. Let $E\subset \R^d$, and $\Gamma\subset\Gamma(d)$ be Borel sets. We have the following:
         \begin{itemize}
     \item[(i)] $\dim_\mathcal{S}\Gamma(E)\geq \dim_\mathcal{H}E+\dim_\mathcal{H}\Gamma-d.$ 
     \item[(ii)] If 
        $\dim_\mathcal{H}E+\dim_\mathcal{H}\Gamma>2d$,
               then $$\mathcal{L}^d(\Gamma(E))>0.$$
    \item[(iii)] For $0<u<d$, if
    $\dim_\mathcal{H}E+\dim_\mathcal{H}\Gamma>d+u$, 
     then $$\dim_{\mathcal{H}} \Gamma(E)\geq u.$$
      \end{itemize}
\end{theorem}

\begin{remark}
\begin{itemize}
    \item[]
    \item[(i)] The bounds in Theorem \ref{thm_main_dilation_translation_general} $(i)$, $(ii)$ are sharp by the examples in Remark \ref{remark_1.2}.
    \item[(ii)] Theorem \ref{thm_main_dilation_translation_general} generalizes Theorem \ref{thm_main_dilation_translation}, which is restricted to the case $\Gamma$ has product structure.
\end{itemize}
\end{remark}

More generally, we also consider the packing problem by multi-parameter dilations and translations as follows. Denote $\Tilde{\Gamma}(d)=\R^d\times \R_{>0}^d$ as the set of all multi-parameter dilations and translations. For each $\tilde{\gamma}=(z,r)=(z_1,\dots, z_d,r_1,\dots,r_d)\in \Tilde{\Gamma}(d)$, and for each $x\in \R^d$ we define
\begin{align*}
    \tilde{\gamma}(x)=(r_1x_1,\dots, r_dx_d)+(z_1,\dots, z_d).
\end{align*}
Let $E\subset \R^d$, and $\Tilde{\Gamma}\subset \Tilde{\Gamma}(d)$ be Borel sets, set
\begin{align*}
    \Tilde{\Gamma}(E)=\{\tilde{\gamma}(x): x\in E, \gamma\in \Tilde{\Gamma}\}.
\end{align*}

Our next main result is the following.
\begin{theorem}\label{thm_main_multi_dilation_translation_general}
    Let $d\geq 1$. Let $E\subset \R^d$, and $\tilde{\Gamma}\subset\Tilde{\Gamma}(d)$ be Borel sets. We have the following:
     \begin{itemize}
     \item[(i)] $\dim_\mathcal{S}\tilde{\Gamma}(E)\geq\dim_\mathcal{H}E+\dim_\mathcal{H}\tilde{\Gamma}-2d+1.$
     \item[(ii)] If 
        $\dim_\mathcal{H}E+\dim_\mathcal{H}\tilde{\Gamma}>3d-1$,
               then $$\mathcal{L}^d(\tilde{\Gamma}(E))>0.$$
    \item[(iii)] For $0<u<d$, if
    $\dim_\mathcal{H}E+\dim_\mathcal{H}\tilde{\Gamma}>2d-1+u$, 
     then $$\dim_{\mathcal{H}} \tilde{\Gamma}(E)\geq u.$$
      \end{itemize}
\end{theorem}

\begin{remark}
 The bounds in Theorem \ref{thm_main_multi_dilation_translation_general} $(i)$, $(ii)$ are sharp. See examples in Section \ref{Section_examples}.
\end{remark}

\subsection{Packing sets by rigid motions}
We shall now investigate the packing problem in Euclidean space by using another class of affine transformations, the set of all rigid motions. 

Let $d\geq 2$, we denote $E(d)=O(d)\times \R^d$ as the Euclidean group or the set of all rigid motions. For each $\theta=(g,z)\in E(d)$, $x\in \R^d$, we define $\theta(x)= g(x)+z$. Let $\Theta\subset E(d)$ and $E\subset \R^d$, we define
\begin{align*}
    \Theta(E)= \bigcup_{\theta\in \Theta}\theta(E)=\{g(x)+z: (g,z)\in \Theta, x\in E\}.
\end{align*}
When $\Theta=G\times Z$, where $G\subset O(d)$ and $Z\subset \R^d$, we denote
\begin{align*}
    GE=\{g(x):  x\in E,g\in G\}\,\, \text{, and }\,\,GE+Z=\{ y+z: y\in GE, z\in Z\}.
\end{align*}
From definition, $\Theta(E)$ is the union of all rotated and translated copies $\theta(E)=g(E)+z$ of $E$, for $\theta=(g,z)\in\Theta$. We will find conditions on the dimensions of $E$ and $\Theta$ to ensure that the set containing all rotated and translated copies $\theta(E)=g(E)+z$ of $E$ for each element $\theta=(g,z)\in \Theta$ has positive Lebesgue measure. 

Before stating the theorem, we will recall the definition of \textit{spherical Fourier dimension} (see \cite{Mattila87}). For a given Borel set $E\subset \R^d$, the spherical Fourier dimension of $E$, denoted by $\dim_{\mathcal{SF}}E$, is defined by
\begin{align*}
    \dim_{\mathcal{SF}}E=\sup\{\alpha \leq d: \exists \mu\in \mathcal{M}(E) \text{ such that } \sigma(\mu)(r)\lesssim r^{-\alpha} \text{ for all }r>0\},
\end{align*}
where $\sigma(\mu)(r)$ is the $L^2$ spherical average of the Fourier transform of $\mu$, see \eqref{eq_def_spherical_average}.
We have the following lemma, see  Section \ref{sec_2} for more details.
\begin{lemma}\label{lemma_spherical_dim}
    Let $E\subset \R^d$ be a Borel set. Then one has $ \dim_\mathcal{H}E\geq \dim_{\mathcal{SF}}E\geq \dim_\mathcal{F}E.$
    Moreover,
    \begin{align*}
    \dim_{\mathcal{SF}}E &= \dim_\mathcal{H}E, \quad \text{ if } \dim_\mathcal{H}E \leq (d-1)/2,\\
    \text{ and }\quad\dim_{\mathcal{SF}}E &\geq \frac{d-1}{d}\dim_\mathcal{H}E.
\end{align*}
\end{lemma}

The main result in this section is the following.
\begin{theorem}\label{thm_main_rigid}
      Let $d\geq 2$. Let $E,Z\subset \R^d$, and $G\subset O(d)$ be Borel sets such that $\dim_{\mathcal{H}}G> \frac{(d-1)(d-2)}{2}$. Then we have the following:
      \begin{itemize}
      \item[(i)] $\dim_{\mathcal F}(GE)\geq \dim_{\mathcal{SF}}E+\dim_{\mathcal{H}}G - \frac{d^2-d}{2}$,
      and
      \begin{align*}
          \dim_\mathcal{S}(GE+Z)\geq \dim_{\mathcal{SF}}E+\dim_{\mathcal{H}}G +\dim_{\mathcal{H}}Z- \frac{d^2-d}{2}.
      \end{align*}
      \item[(ii)] If 
           $
    \dim_{\mathcal{SF}}E+\dim_{\mathcal{H}}G+\dim_{\mathcal{H}}Z> \frac{d^2+d}{2}$,
    then $$\mathcal{L}^d(GE+Z)>0.$$
    \item[(iii)] For $0<u<d$, if 
    $
        \dim_{\mathcal{SF}}E+ \dim_{\mathcal{H}}G+\dim_{\mathcal{H}}Z> \frac{d^2-d}{2}+u,
    $ 
    then $$\dim_{\mathcal{H}}(GE+Z)\geq u.$$
      \end{itemize}
\end{theorem}
Recalling Lemma \ref{lemma_spherical_dim} we have 
\begin{corollary}\label{thm_main_rigidcor}
      Let $d\geq 2$. Let $E,Z\subset \R^d$, and $G\subset O(d)$ be Borel sets such that $\dim_{\mathcal{H}}G> \frac{(d-1)(d-2)}{2}$. Then we have the following:
      \begin{itemize}
      \item[(i)] If 
           $
    \frac{d-1}{d}\dim_{\mathcal{H}}E+\dim_{\mathcal{H}}G+\dim_{\mathcal{H}}Z> \frac{d^2+d}{2}$,
    then $$\mathcal{L}^d(GE+Z)>0.$$
    \item[(ii)] For $0<u<d$, if 
    $
        \frac{d-1}{d}\dim_{\mathcal{H}}E+ \dim_{\mathcal{H}}G+\dim_{\mathcal{H}}Z> \frac{d^2-d}{2}+u,
    $ 
    then $$\dim_{\mathcal{H}}(GE+Z)\geq u.$$
      \end{itemize}
\end{corollary}

\begin{remark}
\begin{itemize}
    \item[] 
    \item[(i)] The conditions in Theorem  \ref{thm_main_rigid} $(i)$, $(ii)$, and Corollary \ref{thm_main_rigidcor} $(i)$ are generally sharp. The sharpness examples are given in Section \ref{Section_examples}.
    \item[(ii)] Corollary \ref{thm_main_rigidcor} also follows from \cite[Theorem 4.3]{Mattila22}.
    \end{itemize}
\end{remark}

Based on the above theorem, it is plausible to make the following conjecture for packing arbitrary Borel sets in Euclidean space using rigid transformations.
\begin{conjecture}
    Let $E\subset \R^d$  and $\Theta\subset E(d)$ be Borel sets. If 
    $$\dim_{\mathcal{SF}}E+\dim_{\mathcal{H}} \Theta>\frac{d^2+d}{2},\quad \text{then}\quad\mathcal{L}^d(\Theta(E))>0.$$
\end{conjecture} 

\subsection{Packing sets by similarity transformations}
Instead of using rigid motions, we can extend the study of the packing problems to the class of similarity transformations.

Let $d\geq 2$, for each $\omega=(g,z,r)\in \Omega(d)=O(d)\times\R^d\times \R_{>0}$, the set of similarity transformations, and $x\in \R^d$, we define $\omega(x)= rg(x)+z$.  Let $\Omega\subset \Omega(d)$ and $E\subset \R^d$, we set
\begin{align*}
    \Omega(E):= \bigcup_{\omega\in \Omega}\omega(E)=\{rg(x)+z: (g,z,r)\in \Omega, x\in E\}.
\end{align*}
When $\Omega=G\times Z\times T$, where $G\subset O(d)$, $Z\subset \R^d$, and $T\subset \R_{>0}$, we denote
\begin{align*}
    GTE=\{rg(x): (g,r)\in G\times T, x\in E\},\,\,\text{ and }\,\,GTE+Z=\{y+z: y\in GTE, z\in Z\}.
\end{align*}
Observe that $\Omega(E)$ is the union of all dilated, rotated, and translated copies $\omega(E)$ of $E$, for each $\omega\in \Omega$. 

Similar to the packing problem using rigid motions, we ask: Under which conditions of the dimensions of $E$ and $\Omega$, we can not pack $\Omega(E)$ into a set of zero Lebesgue measure in $\R^d$? 

The main result in this section is the following.
\begin{theorem}\label{thm_main_similarity}
    Let $d\geq 2$. Let $E,Z\subset \R^d$, $T\subset \R_{>0}$, and $ \Omega\subset \Omega(d)$ be Borel sets. Assume that $\dim_{\mathcal{H}}G>\frac{(d-1)(d-2)}{2}$. Then we have the following:
    \begin{itemize}
    \item[(i)] $\dim_\mathcal{F}(GTE)\geq \dim_{\mathcal{H}} E+\dim_{\mathcal{H}} G+\dim_{\mathcal{H}}T-\frac{d^2-d+2}{2}$, and
    \begin{align*}
        \dim_\mathcal{S}(GTE+Z)\geq \dim_{\mathcal{H}} E+\dim_{\mathcal{H}} G+\dim_{\mathcal{H}}Z+\dim_{\mathcal{H}}T-\frac{d^2-d+2}{2}.
    \end{align*}
        \item[(ii)] If 
           $ \dim_{\mathcal{H}} E+\dim_{\mathcal{H}} G+\dim_{\mathcal{H}}Z+\dim_{\mathcal{H}}T>\frac{d^2+d+2}{2}$,   
    then $$\mathcal{L}^d(GTE+Z)>0.$$
        \item[(iii)] For $0<u<d$, if 
    $\dim_{\mathcal{H}} E+\dim_{\mathcal{H}} G+\dim_{\mathcal{H}}Z+\dim_{\mathcal{H}}T>\frac{d^2-d+2}{2}+u,$
    then $$\dim_{\mathcal{H}}(GTE+Z)\geq u.$$
    \end{itemize}
\end{theorem}

\begin{remark}
\begin{itemize}
    \item[]
    \item[(i)] Theorem \ref{thm_main_similarity} is a fractal variant of results in \cite{Changetal18} by Chang, Cs\"ornyei, H\'era, and Keleti, where the authors considered scaled, rotated, and translated skeletons of polytopes.
    \item[(ii)] The conditions in Theorem \ref{thm_main_similarity} $(i)$, and $(ii)$ are sharp, see examples in Section \ref{Section_examples}.
\end{itemize}
\end{remark}

\subsection{Union of sets by rigid motions in finite fields}
In this section, we will discuss analogous packing results in vector spaces over finite fields.

Let $\F_q^d$ be the $d-$dimensional vector space over a finite field $\F_q$ with $q$ elements, where $q$ is an odd prime power, $d\geq 2$.

For each $x\in \F_q^d, \theta=(g,z)\in O(d)\times \F_q^d$, we define $\theta(x)=gx+z$. For $E\subset \F_q^d$, and $\Theta \subset O(d)\times \F_q^d$, we define
\begin{align*}
    \Theta(E)=\bigcup_{\theta\in \Theta} \theta(E).
\end{align*}

The following three theorems are proved in \cite{Pham23} using bounds on the incidence between points and rigid motions. We refer the reader to \cite{Pham23} for a discussion on the sharpness of these results. 
\begin{theorem}\label{thm_FF_dim2}
    Let $E\subset \F_q^2$ and $\Theta\subset O(2)\times \F_q^2$ with $q\equiv 3\mod 4$. Assume that $|E|^{1/2}|\Theta|\gg q^3$, then we have $|\Theta(E)|\gg q^2$.
    \end{theorem}
    
\begin{theorem}\label{Fur1}
Let $E\subset \F_q^d$ and $\Theta\subset O(d)\times \F_q^d$, with $d\ge 2$. 
We have 
\[\left\vert \Theta(E)\right\vert\gg \min \left\lbrace q^d, ~\frac{|E||\Theta|}{q^d|O(d-1)|}\right\rbrace.\]
\end{theorem}
\begin{theorem}\label{Fur2}
Let $E\subset \F_q^d$ and $\Theta\subset O(d)\times \F_q^d$. Assume in addition that either ($d\ge 3$ odd) or ($d\equiv 2\mod 4$ and $q\equiv 3\mod 4$). 
\begin{enumerate}
    \item[\textup{(1)}] If $|E|<q^{\frac{d-1}{2}}$, then we have 
\[\left\vert \Theta(E)\right\vert\gg \min \left\lbrace q^d, ~\frac{|E||\Theta|}{q^{d-1}|O(d-1)|}\right\rbrace.\]
    \item[\textup{(2)}] If $q^{\frac{d-1}{2}}\le |E|\le q^{\frac{d+1}{2}}$, then we have
\[\left\vert \Theta(E)\right\vert\gg \min \left\lbrace q^d, ~\frac{|\Theta|}{q^\frac{d-1}{2}|O(d-1)|}\right\rbrace.\]
\end{enumerate}
\end{theorem}

We provide simple alternative proofs of these results using some of the ideas from our ${\mathbb R}^d$ results. However, we do not currently have an ${\mathbb R}^2$ variant of Theorem \ref{thm_FF_dim2}. We hope to address this issue in the sequel. 

\vspace{0.5cm}

The structure of the rest of the paper is as follows: In Section \ref{sec_2}, we will give some notations, preliminaries, and lemmas needed for the rest of the paper. The proofs of Theorems \ref{thm_main_dilation_translation}, \ref{thm_main_rigid} and \ref{thm_main_similarity} will be given in Section \ref{sec_3}.  Theorems \ref{thm_main_dilation_translation_general} and \ref{thm_main_multi_dilation_translation_general} will be proved 
in Section \ref{sec_4}.
In Section \ref{sec_5}, we will review some basic background on Fourier analysis over finite fields, and then give the proof of Theorem \ref{thm_FF_dim2}. Some examples related to our results will be presented in Section \ref{Section_examples}.

\section{Preliminaries}\label{sec_2}
Throughout the paper, we will write $A\lesssim_\alpha B$ if $A\leq CB$ where $C>0$ is a constant depending on $\alpha$. If it is clear from the context what $C$ should depend on, we may write only $A\lesssim B$.
If $A\lesssim B$ and $B\lesssim A$, we write $A\approx B$. In the metric space $X$,
the closed ball with center $x$ and radius $r>0$ will be denoted by $B_X(x,t)$, and we write $B(x,t)$ if $X$ is clear in the context. We denote by $\mathcal{L}^d$ the Lebesgue measure in the Euclidean space $\R^d$, $d\geq 1$. The orthogonal group of $\R^d$ is $O(d)$, and its Haar probability measure is $dg$. The set of all rigid motions in $\R^d$ is denoted by $E(d)=O(d)\times \R^d$, and $\Omega(d)=O(d)\times \R^d\times \R_{>0}$ stands for the set of all similarity transformations in $\R^d$. 
Let $A\subset \R^d$ ($O(d), E(d)$, or $\Omega(d)$). 
We denote $\mathcal{M}(A)$ as the set of non-zero Radon measures $\mu$ on $\R^d$ with compact support $\spt \mu \subset A$. The Hausdorff dimension and Fourier dimension of $A$ will be denoted by $\dim_{\mathcal{H}}A$ and $\dim_{\mathcal{F}}A$, resprectively. The Fourier transform of $\mu$ is defined by
\begin{align*}
    \widehat{\mu}(\xi)=\int e^{-2\pi i \xi\cdot x}\,d\mu(x), \quad \xi\in \R^d.
\end{align*}

\subsection{Frostman's lemma, dimensions of sets, ball averages, and spherical averages}
\begin{lemma}[Frostman's lemma, Theorem 2.7, \cite{Mattila15}]
Let $0\leq s\leq d$.
For a Borel set $E\subset \R^d$, the $s-$dimensional Hausdorff measure of $E$ is positive if and only if there exists a measure $\mu\in\mathcal{M}(E)$ satisfying
\begin{align*}
    \mu(B(x,t))\lesssim r^s, \quad \forall\, x\in \R^d,\,\,r>0.
\end{align*}
\end{lemma}
In particular, Frostman's lemma implies that given any exponent $0<s< \dim_{\mathcal{H}}(E)$, there exists a probability measure $\mu$ on $E$ such that 
\begin{align}\label{eq_Frostman_measure}
    \mu(B(x,t))\lesssim r^{s}, \quad \forall\, x\in \R^d, r>0.
\end{align}
A measure satisfying condition \eqref{eq_Frostman_measure} is often called an $s$-dimensional Frostman measure (or $s$-Frostman measure).
The $s$-energy integral of a measure $\mu\in \mathcal{M}(\R^d)$ (see \cite{Mattila95,Mattila15})
is
\begin{align*}
    I_s(\mu)=\iint|x-y|^{-s}\,d\mu(x)d\mu(y) = c(n,s)\int|\widehat{\mu}(\xi)|^2|\xi|^{s-d}d\xi.
\end{align*}
If $\mu\in \mathcal{M}(\R^d)$ satisfies the Frostman condition $\eqref{eq_Frostman_measure}$, then $I_t(\mu)<\infty$ for all $0<t<s$.

We have for any Borel set $A\subset \R^d$ with Hausdorff dimension $\dim_{\mathcal{H}}A>0$, (see Theorem 8.9 in \cite{Mattila95}),
\begin{align*}
    \dim_{\mathcal{H}}A
    &= \sup\{ s\leq d: \exists \mu\in \mathcal{M}(A) \text{ such that } \mu(B(x,r))\leq r^s \text{ for }x\in \R^d,r>0\}\\
    &=\sup\{ s\leq d: \exists \mu\in \mathcal{M}(A) \text{ such that } I_s(\mu)<\infty\}.
\end{align*}
The Fourier dimension of a set $A\subset \R^d$ is
\begin{align*}
    \dim_{\mathcal{F}}A=\sup\{ s\leq d: \exists \mu\in  \mathcal{M}(A) \text{ such that } |\widehat{\mu}(x)| \lesssim (1+|x|)^{-\frac{s}{2}}, \text{ for all }x\in \R^d\}.
\end{align*}
The Sobolev dimension of a measure $\mu\in \mathcal{M}(\R^d)$ is
\begin{align*}
    \dim_{\mathcal{S}}\mu=\sup\{ s\in \R:  \int |\widehat{\mu}(x)|^2(1+|x|)^{s-d}dx<\infty\}.
\end{align*}
We will say that a set $A\subset \R^d$ has Sobolev dimension $\dim_\mathcal{S}A\geq s$ if $A$ carries a Borel probability measure $\mu$ such that $\dim_\mathcal{S}\mu\geq s$. The greater the Sobolev dimension is, the smoother the measure is in some sense. We recall the following well-known result, see \cite{Mattila15}.
\begin{proposition}[{\cite[Theorem 5.4]{Mattila15}}]\label{prop_sobolev_dim}
    Let $\mu\in \mathcal{M}(\R^d)$.
    \begin{itemize}
        \item[(i)] If $0<\dim_\mathcal{S}\mu<d$, then $\dim_{\mathcal{S}}\mu=\sup\{ s>0: I_s(\mu)<\infty\}$.
        \item[(ii)] If $\dim_\mathcal{S}\mu>d$, then $\mu\in L^2(\R^d)$.
        \item[(iii)] If $\dim_\mathcal{S}\mu>2d$, then $\mu$ is a continuous function. 
    \end{itemize}
\end{proposition}

Let $\mu\in \mathcal{M}(\R^d)$ be an $s$-Frostman measure. We
have the following ball average estimate (see \cite[Section 3.8]{Mattila15}),
\begin{align}\label{ball_averages}
    \int_{B(0,R)}|\widehat{\mu}(\xi)|^2d\xi\lesssim R^{d-s},\quad R>0.
\end{align}
Next, given a Radon measure $\mu$ with compact support on $\R^d$, $d\geq 2$, we define the $L^2$ spherical averages of the Fourier transform of $\mu$ by
\begin{align}\label{eq_def_spherical_average}
    \sigma(\mu)(r)=\int_{S^{d-1}}\vert \widehat{\mu}(rv)\vert^2d\sigma(v), \quad r>0,
\end{align}
where $\sigma$ is the surface measure on the unit sphere $S^{d-1}$. In \cite{Mattila87}, Mattila developed a method to study Falconer's distance problem by studying the decay rates of spherical averages of fractal measures, namely, the supremum of the numbers $\beta$ for which for all $r>1$, one has
\begin{align}\label{eq_decay_spherical_average}
    \sigma(\mu)(r)\lesssim r^{-\beta}.
\end{align}
For any $s$-Frostman measure $\mu$ we have for all $\e>0$, $r>1$,
\begin{equation}\label{eq_spherical_average}
    \sigma(\mu)(r)\lesssim
    \left\{
    \begin{array}{lll}
    r^{-s+\e}, &s\in \big(0,\frac{d-1}{2}\big],\\
    r^{-\frac{d-1}{2}+\e}, &s\in \big[\frac{d-1}{2},\frac{d}{2}\big], \\
    r^{-\frac{(d-1)s}{d}+\e}, &s\in (0,d).
    \end{array}
    \right.
\end{equation}
The first two estimates were proved in \cite{Mattila87}. The last estimate is the deepest. It is due to Wolff \cite{Wolff99} for $d=2$, to Du-Guth-Ou-Wang-Wilson-Zhang \cite{Duetal21} for $d=3, 3/2<s\leq 2$, and to Du-Zhang \cite{DuZhang19} for $d=3, 2\leq s <3$, and $d\geq 4$. The first estimate is always sharp and for $d=2$ they all are sharp. 

The lower bounds for the spherical Fourier dimension in Lemma \ref{lemma_spherical_dim} follow from the best known decay of spherical averages \eqref{eq_spherical_average}.
%

\subsection{Some lemmas}
In this section, we will recall some results needed for the proof of the theorems. 

First, we give a proof for Theorem \ref{thm_B}.
\begin{proof}[Proof of Theorem \ref{thm_B}]
    Let $\alpha < \dim_{\mathcal{F}}A$ and $\beta < \dim_{\mathcal{H}}B$, $\mu\in \mathcal{M}(A), \nu\in \mathcal{M}(B)$ such that $|\widehat{\mu}(\xi)|^2\leq |\xi|^{-\alpha}$ and $I_{\beta}(\nu)<\infty$. Then one has $\mu\ast \nu\in \mathcal{M}(A+B)$ and 
    \begin{align*}
        \int|\widehat{\mu\ast \nu}(\xi)|^2|\xi|^{\alpha+\beta-d}d\xi\leq
    \int|\widehat{\nu}(\xi)|^2|\xi|^{\beta-d}d\xi=cI_{\beta}(\nu)<\infty.
    \end{align*}
    This gives $(i)$. Parts $(ii)$ and $(iii)$ follow from $(i)$ and Proposition \ref{prop_sobolev_dim}.
\end{proof}
Next, we recall the following lemmas. The proofs can be found in \cite{Mattila17}, and \cite{Mattila22}.
\begin{lemma}[\cite{Mattila22}, Lemma 3.1]\label{lemma_1}
    Let $\gamma\in \mathcal{M}(O(d))$ be an $\alpha$-dimensional Frostman measure, $\alpha>\frac{(d-1)(d-2)}{2}$, and put $\beta=\alpha-\frac{(d-1)(d-2)}{2}$. Then for $x,y\in \R^d\setminus\{0\}$, $\delta>0$,
    \begin{align}\label{eq_Mat1}
        \gamma\big(\{ g: \vert g(y)-x\vert < \delta\}\big)\lesssim \min \bigg\{ \bigg(\frac{\delta}{\vert y\vert }\bigg)^{\beta},\bigg(\frac{\delta}{\vert x\vert}\bigg)^{\beta}\bigg\}.
    \end{align}
\end{lemma}
Let $\gamma\in \mathcal{M}(O(d))$, and
$\mu\in \mathcal{M}(\R^d)$. We define
\begin{align*}
    \sigma_\gamma(\mu)(\xi)=\int|\widehat{\mu}(g^{-1}(\xi))|^2d\gamma(g),\quad \xi\in \R^d.
\end{align*}
The following lemma is the key to relating Hausdorff dimensions of the sets $E\subset \R^d$ and $G\subset O(d)$ to packing problems in many cases involving rotations. This lemma follows from the estimates for the decay rates of the spherical averages \eqref{eq_spherical_average}.

\begin{lemma}[\cite{Mattila22}, Lemma 4.1]\label{lemma_average}
    Let $\gamma\in \mathcal{M}(O(d))$ be a measure which satisfies condition \eqref{eq_Mat1} with some exponent $\beta\in (0,d-1]$. Assume that $\mu\in \mathcal{M}(\R^d)$ is such that
    \eqref{eq_decay_spherical_average} holds for exponent $s_\mu\leq d$. Then
    for $\xi\in \R^d$ with $\vert \xi\vert>1$, and for $\e>0$, we have
    \begin{align*}
        \sigma_\gamma(\mu)(\xi)\lesssim \vert \xi\vert^{-(s_\mu+\beta+1-d-\e)}.
    \end{align*}
\end{lemma}

Next, we will give modifications
of Lemma \ref{lemma_average}. Let $\mu\in\mathcal M(\R^d)$, $\gamma\in\mathcal M(O(d))$, and $\zeta\in \mathcal M(\R_{>0})$. For $\xi\in \R^d$, we define
\begin{align*}
    \tilde{\sigma}_{\zeta}(\mu)(\xi)&=\int|\widehat{\mu}(r\xi)|^2\,d\zeta(r),\\
    \sigma_{\gamma,\zeta}(\mu)(\xi)&=\iint|\widehat{\mu}(rg^{-1}(\xi))|^2\,d\gamma(g)\,d\zeta(r).
\end{align*}

\begin{lemma}\label{lemma_average_2}
    Let $\gamma\in \mathcal{M}(O(d))$ be a measure which satisfies condition \eqref{eq_Mat1} with some exponent $\beta\in (0,d-1]$. Assume that $\mu\in \mathcal{M}(\R^d)$, $\zeta\in \mathcal{M}(\R_{>0})$ are Frostman measures with exponents $s_\mu$ and $s_{\zeta}$, respectively. Then for $\xi\in \R^d$ with $\vert \xi\vert>1$, and for $0<\e<1$,
    we have
    \begin{align*}
        \sigma_{\gamma,\zeta}(\mu)(\xi)\lesssim |\xi|^{-(s_\mu+s_\zeta+\beta-d-\e)}.
    \end{align*}
\end{lemma}

\begin{lemma}\label{lemma_average_3}
    Let $\mu\in \mathcal{M}(\R^d)$, $\zeta\in \mathcal{M}(\R_{>0})$ be
      Frostman measures with exponents $s_\mu$ and $s_{\zeta}$, respectively. Then for $\xi\in \R^d$ with $\vert \xi\vert>1$, and for $0<\e<1$,
    we have
    \begin{align*}
        \tilde{\sigma}_{\zeta}(\mu)(\xi)\lesssim |\xi|^{-(s_\mu+s_{\zeta}-d-\e)}.
    \end{align*}
\end{lemma}

\begin{proof}[Proof of Lemma \ref{lemma_average_2}]
Without loss of generality, we assume that $\spt \zeta\subset [1,2]$.
    Let $\phi$ be a smooth compactly supported function such that $\phi\geq 0$, and $\phi = 1$ on $\spt\mu$. Then $\widehat{\mu}=\widehat{\phi\mu}=\widehat{\phi}\ast\widehat{\mu}$. Thus for $|\xi|>1$, one can write
    \begin{align*}
        \sigma_{\gamma,\zeta}(\xi)
&=\iint|\widehat{\phi\mu}(rg^{-1}(\xi))|^2\,d\gamma(g)\,d\zeta(r)\\
&=\iint\bigg|\int\widehat{\phi}(rg^{-1}(\xi)-x)\widehat{\mu}
(x)\,dx\bigg|^2  \,d\gamma(g)\,d\zeta(r).
    \end{align*}
    By applying the Cauchy-Schwarz inequality and the fast decay property of $\widehat{\phi}$, the integral is dominated by
\begin{align*}
&\lesssim\iint\bigg(\int|\widehat{\phi}(rg^{-1}(\xi)-x)|dx\bigg)\cdot \bigg(\int|\widehat{\phi}(rg^{-1}(\xi)-x)||\widehat{\mu}
(x)|^2\,dx \bigg) \,d\gamma(g)\,d\zeta(r)\\
&\lesssim \iiint_{\{|rg^{-1}(\xi)-x|
\leq|\xi|^{\e}\}}|\widehat{\mu}
(x)|^2\,d\gamma(g)\,d\zeta(r)\,dx\\
&\hspace{2cm}+\sum_{j=1}^{\infty}\iiint_{\{|\xi|^{\e j}\leq|rg^{-1}(\xi)-x|
\leq|\xi|^{\e(j+1)}\}}|\widehat{\phi}(rg^{-1}(\xi)-x)||\widehat{\mu}
(x)|^2\,d\gamma(g)\,d\zeta(r)\,dx.\\
&=I_1(\xi)+I_2(\xi).
\end{align*}
To estimate $I_1(\xi)$, note that if $|rg^{-1}(\xi)-x|\leq|\xi|^{\e}$, then $\big|r-|x|/|\xi|\big|\leq|\xi|^{\e-1}$, whence, by \eqref{eq_Mat1},
\begin{align*}
&\gamma\times\zeta(\{(g,r):|rg^{-1}(\xi)-x|\leq|\xi|^{\e}\})\\
&=\int_{\{r:|r-|x|/|\xi||\leq|\xi|^{\e-1}\}}\gamma(\{g:|rg^{-1}(\xi)-x|\leq|\xi|^{\e}\})\,d\zeta(r)\\
&\lesssim \int_{\{r:|r-|x|/|\xi||\leq|\xi|^{\e-1}\}} |\xi|^{\beta(\e-1)}d\zeta(r)\\
&\lesssim |\xi|^{(\e-1)(\beta+s_{\zeta})}.
\end{align*}
Hence the first integral is bounded by
$$I_1(\xi)\lesssim |\xi|^{(\e-1)(\beta+s_{\zeta})}\int_{\{|x|\leq 3|\xi|\}}|\widehat{\mu}
(x)|^2\,d(x)\lesssim |\xi|^{(\e-1)(\beta+s_{\zeta})+d-s_\mu},$$
where in the last inequality, we used the ball average estimate \eqref{ball_averages}.

For $I_2(\xi)$, we have, again by the fast decay of $\widehat{\phi}$, and the ball average estimate \eqref{ball_averages},
\begin{align*}
    I_2(\xi)
    &\lesssim \sum_{j=1}^{\infty}\iiint_{\{|\xi|^{\e j}\leq|rg^{-1}(\xi)-x|
\leq|\xi|^{\e(j+1)}\}}|\xi|^{-N\e j}\,d\gamma(g)\,d\zeta(r)|\widehat{\mu}(x)|^2\,dx\\
& \lesssim \sum_{j=1}^{\infty}\int_{\{|x|\leq 3|\xi|^{j+1}\}}|\xi|^{-N\e j} \gamma\times \zeta\big(\{(g,r):|rg^{-1}(\xi)-x|
\leq|\xi|^{\e(j+1)}\}\big)|\widehat{\mu}(x)|^2\,dx\\
& \lesssim \sum_{j=1}^{\infty} \int_{\{|x|\leq 3|\xi|^{j+1}\}}|\xi|^{-N\e j}|\xi|^{(\e(j+1)-1)(\beta+s_{\zeta})}|\widehat{\mu}(x)|^2\,dx\\
& \lesssim \sum_{j=1}^{\infty} |\xi|^{-N\e j}|\xi|^{(\e(j+1)-1)(\beta+s_{\zeta})}|\xi|^{(j+1)(d-s_\mu)}\\
& \lesssim |\xi|^{d-\beta-s_{\zeta}-s_\mu+\e}\sum_{j=1}^{\infty} |\xi|^{-j(N \e-d+s_\mu -(\beta+s_{\zeta})\e)}\\
&\lesssim |\xi|^{d-\beta-s_{\zeta}-s_\mu+\e},
\end{align*}
provided $N$ is chosen big enough so that $N \e+s_\mu>(\beta+s_{\zeta})\e+d$. The lemma 
then follows.
\end{proof}
\begin{proof}[Proof of Lemma \ref{lemma_average_3}]
Without loss of generality, we assume that $\spt \zeta\subset [1,2]$.
    Let $\phi$ be a smooth compactly supported function such that $\phi\geq 0$, and $\phi = 1$ on $\spt\mu$. Then $\widehat{\mu}=\widehat{\phi\mu}=\widehat{\phi}\ast\widehat{\mu}$. Thus for $|\xi|>1$, one can write
    \begin{align*}
        \tilde{\sigma}_{\zeta}(\xi)
&=\int|\widehat{\phi\mu}(r\xi)|^2\,d\zeta(r)=\int\bigg|\int\widehat{\phi}(r\xi-x)\widehat{\mu}
(x)\,dx\bigg|^2  \,d\zeta(r).
    \end{align*}
       By applying the Cauchy-Schwarz inequality and the fast decay property of $\widehat{\phi}$, the integral is dominated by
\begin{align*}
&\lesssim\int\bigg(\int|\widehat{\phi}(r\xi-x)|dx\bigg)\cdot \bigg(\int|\widehat{\phi}(r\xi-x)||\widehat{\mu}
(x)|^2\,dx \bigg) \,d\zeta(r)\\
&\lesssim \iint_{\{|r\xi-x|
\leq|\xi|^{\e}\}}|\widehat{\mu}
(x)|^2\,d\zeta(r)\,dx\\
&\hspace{2cm}+\sum_{j=1}^{\infty}\iint_{\{|\xi|^{\e j}\leq|r\xi-x|
\leq|\xi|^{\e(j+1)}\}}|\widehat{\phi}(r\xi-x)||\widehat{\mu}
(x)|^2\,d\zeta(r)\,dx.\\
&=I_1(\xi)+I_2(\xi).
\end{align*}
To estimate $I_1(\xi)$, note that if $|r\xi-x|\leq|\xi|^{\e}$, then $\big|r-|x|/|\xi|\big|\leq|\xi|^{\e-1}$, whence by assumption on $\zeta$, one has
\begin{align*}
\zeta\big(\{r:|r\xi-x|\leq|\xi|^{\e}\}\big)
\lesssim \zeta \big(B(|x|/|\xi|,|\xi|^{\e-1})\big)\lesssim |\xi|^{(\e-1)s_{\zeta}}.
\end{align*}
Hence the first integral is bounded by 
$$I_1(\xi)\lesssim |\xi|^{(\e-1)s_{\zeta}}\int_{\{|x|\leq 3|\xi|\}}|\widehat{\mu}
(x)|^2\,d(x)\lesssim |\xi|^{d-s_\mu-(1-\e)s_\zeta},$$
where in the last inequality, we used the ball average estimate \eqref{ball_averages}.

For $I_2(\xi)$, we have, again by the fast decay of $\widehat{\phi}$, and the ball average estimate \eqref{ball_averages},
\begin{align*}
    I_2(\xi)
    &\lesssim \sum_{j=1}^{\infty}\iint_{\{|\xi|^{\e j}\leq|r\xi-x|
\leq|\xi|^{\e(j+1)}\}}|\xi|^{-N\e j} \,d\zeta(r)|\widehat{\mu}(x)|^2\,dx\\
& \lesssim \sum_{j=1}^{\infty}\int_{\{|x|\leq 3|\xi|^{j+1}\}}|\xi|^{-N\e j}  \zeta\big(\{r:|r\xi-x|
\leq|\xi|^{\e(j+1)}\}\big)|\widehat{\mu}(x)|^2\,dx\\
& \lesssim \sum_{j=1}^{\infty} \int_{\{|x|\leq 3|\xi|^{j+1}\}}|\xi|^{-N\e j}|\xi|^{(\e(j+1)-1)s_{\zeta}}|\widehat{\mu}(x)|^2\,dx\\
& \lesssim \sum_{j=1}^{\infty} |\xi|^{-N\e j}|\xi|^{(\e(j+1)-1)s_{\zeta}}|\xi|^{(j+1)(d-s_\mu)}\\
& \lesssim |\xi|^{d-s_{\zeta}-s_\mu+\e}\sum_{j=1}^{\infty} |\xi|^{-j(N \e-d+s_\mu-s_{\zeta}\e)}\\
&\lesssim |\xi|^{d-s_{\zeta}-s_\mu+\e}
\end{align*}
provided $N$ is chosen big enough so that $N \e+s_\mu>s_{\zeta}\e+d$. The lemma 
then follows.
\end{proof}

\section{Proofs of Theorems \ref{thm_main_dilation_translation}, \ref{thm_main_rigid}, and \ref{thm_main_similarity}}\label{sec_3}
In this section, we will give the proofs of our main results for the packing problems in Euclidean space using affine transformations, assuming the above sets of transformations have product structure.

\subsection{Proof of Theorem \ref{thm_main_dilation_translation}}
We will prove part $(i)$. Combining part $(i)$ and Theorem \ref{thm_B}, we get parts $(ii)$ and $(iii)$. 

Without loss of generality, we may assume that $T\subset [1,2]$.

To prove part $(i)$, let $0<s_E<\dim_\mathcal{H} E$, $0<s_T<\dim_\mathcal{H} T$. Let $\mu$ and $\zeta$ be Frostman measures on $E$ and $T$, respectively, with exponents $s_E$ and $s_T$.

Define a measure $\nu$ supported on $TE$ by relation
$$\int f(u)\,d\nu(u)=\iint f(rx)\,d\mu(x)\,d\zeta(r),\quad \forall f\in C_0(\R^d).$$
In other words, $\nu$ is the push forward measure of $\mu\times\zeta$ under the map $(x,r)\mapsto rx$.

For $\xi\in \R^d$, the Fourier transform of $\nu$ at $\xi$ is given by
\begin{align*}
    \widehat{\nu}(\xi)
&=
\iint e^{-2\pi i \xi\cdot(rx)}\,d\mu(x)\,d\zeta(r)=\int\widehat{\mu}
(r\xi)\,d\zeta(r).
\end{align*}
Hence, by Cauchy-Schwarz inequality, one has
\begin{align*}
    |\widehat{\nu}(\xi)|^2 
    \lesssim  \int|\widehat{\mu}
(r\xi)|^2\,d\zeta(r) =  \tilde{\sigma}_{\zeta}(\xi).
\end{align*}
Let $0<\e<1$, and apply Lemma \ref{lemma_average_3}, we obtain 
\begin{align}
    |\widehat{\nu}(\xi)|^2 
    &\lesssim (1+|\xi|)^{-(s_E+s_T-d-\e)},  \quad \forall \xi\in \R^d.
\end{align}
By the definition of Fourier dimension, this implies that
\begin{align*}
    \dim_\mathcal{F}(TE)\geq s_E+s_T-d-\e.
\end{align*}
It can be seen easily that when letting $s_E\to \dim_\mathcal{H}E$, $s_T\to \dim_\mathcal{H}T$, and $\e\to 0$, one gets
\begin{align*}
    \dim_\mathcal{F}(TE)\geq \dim_\mathcal{H}E+\dim_\mathcal{H}T-d.
\end{align*}
This completes the proof of the theorem.
\begin{flushright}
$\Box$
\end{flushright}

\subsection{Proof of Theorem \ref{thm_main_rigid}}
We will give the proof of part $(i)$, as combining $(i)$ with Theorem \ref{thm_B}, we obtain parts $(ii)$ and $(iii)$.

     To prove part $(i)$, we choose $0<s_E<\dim_{\mathcal{SF}}E$, $\frac{(d-1)(d-2)}{2}<s_G<\dim_{\mathcal{H}}G$.
     
     Let $\mu\in \mathcal{M}(E)$ be a Radon measure supported on $E$ such that
     \begin{align*}
         \sigma(\mu)(r)\lesssim r^{-s_E}, \quad\forall r>0.
     \end{align*}
     Let $\gamma$ be Frostman measure supported on $G$ with exponent $s_G$.
     
     We define a measure
    $\nu\in \mathcal{M}(GE)$ by the following relation
     \begin{align*}
         \int f(u)\,du=\iint f(gx)\,d\mu(x)\,d\gamma(g),\quad \forall f\in C_0(\R^d).
     \end{align*}
    In other words, $\nu$ is the push forward measure of $\mu\times\gamma$ under the map $(x,g)\mapsto g(x)$. 

    The Fourier transform of $\nu$ at $\xi\in \R^d$ is given by
    \begin{align*}
        \widehat{\nu}(\xi)
        &=\iint e^{-2\pi i \xi\cdot g(x)}\,d\mu(x)d\gamma(g)=\int \widehat{\mu}(g^{-1}(\xi))d\gamma(g).
    \end{align*}
    Thus, by invoking Cauchy-Schwarz inequality, one has
    \begin{align*}
         \vert \widehat{\nu}(\xi)\vert^2 
        \lesssim  \int |\widehat{\mu}(g^{-1}(\xi))|^2 d\gamma(g)
        =  \sigma_\gamma(\mu)(\xi).
    \end{align*}
    Choose $\e>0$, and then applying Lemma \ref{lemma_average} with $\beta=s_G-\frac{(d-1)(d-2)}{2}$, we have
    \begin{equation*}
    \begin{split}
         \vert \widehat{\nu}(\xi)\vert^2 
        \lesssim \sigma_\gamma(\mu)(\xi)
        \lesssim    
        (1+|\xi|)^{-(s_E+s_G-\frac{d(d-1)}{2}-\epsilon)},
        \quad\quad
        \forall \xi\in \R^d,      
    \end{split}
    \end{equation*}
    which yields that 
    \begin{align*}
        \dim_\mathcal{F}(GE)\geq s_E+s_G-\frac{d(d-1)}{2}-\epsilon.
    \end{align*}
    Let $s_E\to \dim_\mathcal{SF}E$, $s_G\to \dim_\mathcal{H}G$, and $\e\to 0$, we find that
    \begin{align*}
        \dim_\mathcal{F}(GE)\geq \dim_{\mathcal{SF}}E+\dim_\mathcal{H}G-\frac{d(d-1)}{2}.
    \end{align*}
    This finishes the proof of the theorem.
\begin{flushright}
$\Box$
\end{flushright}

\subsection{Proof of Theorem \ref{thm_main_similarity}}
We will give the proof of part $(i)$. For parts $(ii)$ and $(iii)$, the proof follows by combining $(i)$ with Theorem \ref{thm_B}.

Without loss of generality, we may assume that $T\subset [1,2]$. Let $0<s_E<\dim_\mathcal{H}E$, $0<\frac{(d-1)(d-2)}{2}<s_G<\dim_\mathcal{H}G$, and $0<s_T<\dim_\mathcal{H}T$.
Let $\mu, \gamma$, and $\zeta$ be Frostman measures on $E, G$ and $T$, respectively, with exponents $s_E,s_G$ and $s_T$.

Define a measure $\nu$ supported on $GTE$ by relation
$$\int f(u)\,d\nu(u)=\iiint f(rg(x))\,d\mu(x)\,d\gamma(g)\,d\zeta(r),\quad \forall f\in C_0(\R^d).$$
We have for $\xi\in \R^d$,
$$\widehat{\nu}(\xi)=\iint\widehat{\mu}
(rg^{-1}(\xi))\,d\gamma(g)\,d\zeta(r).$$
By Cauchy-Schwarz inequality, one has
\begin{equation*}\label{eq3}
   |\widehat{\nu}(\xi)|^2 \lesssim\iint |\widehat{\mu}(rg^{-1}(\xi))|^2d\gamma(g)d\zeta(r)= \sigma_{\gamma,\zeta}(\xi).
\end{equation*}
Let $0<\e<1$, and apply Lemma \ref{lemma_average_2} with $\beta=s_G-(d-1)(d-2)/2$, we obtain
\begin{align*}
    |\widehat{\nu}(\xi)|^2 
    &\lesssim (1+|\xi|)^{-(s_E+s_G+s_T-\frac{(d-1)(d-2)}{2}-d-\e)},\quad\quad \forall \xi\in \R^d.
\end{align*}
Hence we find that
\begin{align*}
    \dim_\mathcal{F}(GTE)\geq s_E+s_G+s_T-\frac{d^2-d+2}{2}-\e.
\end{align*}
Letting $s_E\to \dim_\mathcal{H}E$, $s_G\to\dim_\mathcal{H}G$, $s_T\to \dim_\mathcal{H}T$, and $\e\to 0$, we conclude that
\begin{align*}
    \dim_\mathcal{F}(GTE)\geq \dim_\mathcal{H}E+\dim_\mathcal{H}G+\dim_\mathcal{H}T-\frac{d^2-d+2}{2}.
\end{align*}
This completes the proof of the theorem.
\begin{flushright}
$\Box$
\end{flushright}

\section{Proofs of Theorems \ref{thm_main_dilation_translation_general} and \ref{thm_main_multi_dilation_translation_general}}\label{sec_4}

In this section, we will give proofs of Theorems \ref{thm_main_dilation_translation_general} and \ref{thm_main_multi_dilation_translation_general}. First, we will recall a result of D. Oberlin \cite{Oberlin14} on the exceptional estimate for projections.

Let $0\leq l\leq n$ be integers. For each $x\in \R^{l(n-l)}$, we define the projection $P_x: \R^n\to \R^l$ by
\begin{align*}
    P_x(y_1,\dots,y_n)=(y_1,\dots, y_l)+X\cdot (y_{l+1},\dots, y_n)^T=\Tilde{y}+X\cdot \Bar{y}.
\end{align*}
Here $X=(x_{ij})\in M_{l\times (n-l)}$ is the matrix identified by $x$, $\Tilde{y}=(y_1,\dots, y_l)$ and $\Bar{y}=(y_{l+1},\dots, y_n)$.

With $x$ and $X$ as above, for each $\xi\in \R^l$, we define the map $T_x:\R^l\to \R^{n-l}$ by
\begin{align*}
    T_x(\xi)=X^T\cdot \xi^T.
\end{align*}

\begin{theorem}[{\cite[Theorem 1.2]{Oberlin14}}]\label{thm_Oberlin_projection}
    Suppose $\lambda$ is a compactly supported nonnegative Borel measure on $\R^{l(n-l)}$ which satisfies the condition
\begin{align}\label{eq_condition_measure}
    \lambda(\{ x\in \R^{l(n-l)}: |T_x(\xi)-p|\leq \delta\})\leq c\delta^\beta,
\end{align}
for some $c>0$ and all $\xi\in \R^l$ with $|\xi|=1$, $p\in \R^{n-l}$, and $\delta>0$. Suppose $E\subset \R^n$ is a Borel set with Hausdorff dimension at least $\alpha$. Suppose
\begin{align}\label{eq_condition_dim}
    n-l+\sigma-\alpha<\beta.
\end{align}
Then for $\lambda$-almost all $x\in \R^{l(n-l)}$, $\dim_\mathcal{S}P_x(E)\geq \sigma$.
\end{theorem}

\subsection{Proof of Theorem \ref{thm_main_dilation_translation_general}}
We will give the proof for part $(i)$. Parts $(ii)$ and $(iii)$ follow from $(i)$ and Proposition \ref{prop_sobolev_dim}.

    Let $0<s_E<\dim_\mathcal{H}E$ and $0<s_\Gamma<\dim_\mathcal{H}\Gamma$. Let $\mu$ be a Frostman measure supported on $E$ with exponent $s_E$. Furthermore, without loss of generality, we may assume that $\spt \mu\subset E\subset B(0,1)$.
    
    We will define a family of projections as follows. Put $n=d+1$, and $l=d$. For $x\in \R^{d}$, we define the map
\begin{align*}
    P_x: \R^{d+1}&\to \R^d\\
    (z,r)&\mapsto rx+z=z+Xr.
\end{align*}
where $X=\begin{pmatrix}
    x_1\\
    \vdots\\
    x_d
\end{pmatrix}\in M_{d\times 1}$.
Then we can write
\begin{align*}
    \Gamma(E)=\{rx+z: x\in E, (z,r)\in \Gamma\}=\bigcup_{x\in E}P_x(\Gamma).
\end{align*}
We will show that if
\begin{align*}
    \dim_\mathcal{H}E+\dim_\mathcal{H}\Gamma> d+u,
\end{align*}
then $\dim_\mathcal{S}\Gamma(E)\geq u$. 

For each $x$, we define $T_x: \R^d\to \R$ by
\begin{align*}
    T_x(\xi)=X^T\cdot \xi^T=
    \begin{pmatrix}
        x_1&\cdots&x_d
    \end{pmatrix}
    \cdot
    \begin{pmatrix}
        \xi_1\\
        \vdots\\
        \xi_d
    \end{pmatrix}=\xi_1x_1+\cdots+\xi_dx_d.
\end{align*}
To apply Theorem \ref{thm_Oberlin_projection}, we need to verify that condition \eqref{eq_condition_measure} holds with measure $\mu$  for some exponent $\beta$. 

Let $\xi\in \R^d$, $|\xi|=1$, $p\in \R$, and $\delta>0$. Observe that
\begin{align*}
    |T_x(\xi)-p|\leq \delta \quad\Longrightarrow \quad |\xi_1x_1+\cdots+\xi_dx_d-p|\leq \delta.
\end{align*}
Put $H_\xi=\{x\in \R^d: \xi_1x_1+\cdots+\xi_dx_d=p\}$, then $H_\xi$ is a $(d-1)$-dimensional affine space in $\R^d$. Hence
\begin{align*}
    \{x\in E: |T_x(\xi)-p|\leq \delta\}
    &\subset \{x\in E: |\xi_1x_1+\cdots+\xi_dx_d-p|\leq \delta\}\subset H_\xi(\delta)\cap B(0,1),
\end{align*}
where $A(\epsilon)$ denotes the $\epsilon$-neighborhood of $A$.
Since $H_\xi$ has dimension $d-1$, we can cover $H_\xi(\delta)\cap B(0,1)$ by $\lesssim\delta^{-(d-1)}$ balls of radius $\delta$ with bounded overlaps. By Frostman condition, we have
\begin{align*}
    \mu\big(\{x\in E: |T_x(\xi)-p|\leq \delta\}\big)\lesssim \delta^{-(d-1)}\delta^{s_E}=\delta^{-d+1+s_E}.
\end{align*}
Thus \eqref{eq_condition_measure} holds with $\beta=-d+1+s_E$.

Apply Theorem \ref{thm_Oberlin_projection}, we find that if
\begin{align*}
    (d+1)-d+u<s_\Gamma+1+s_E-d \quad \Longleftrightarrow \quad s_E+s_\Gamma>d+u,
\end{align*}
then for $\mu$-almost all $x\in \R^d$, 
\begin{align*}
    \dim_\mathcal{S}P_x(\Gamma)\geq u.
\end{align*}
This finishes the proof of the theorem.
\begin{flushright}
$\Box$
\end{flushright}

\subsection{Proof of Theorem \ref{thm_main_multi_dilation_translation_general}}
    Let $0<s_E<\dim_\mathcal{H}E$ and $0<s_{\tilde{\Gamma}}<\dim_\mathcal{H}\tilde{\Gamma}$. Let $\mu$ be a Frostman measure supported on $E$ with exponent $s_E$. Furthermore, without loss of generality, we may assume that $\spt \mu\subset E\subset B(0,1)$.
    
    We will define a family of projections as follows. Put $n=2d$, and $l=d$. For each $x\in \R^{d}$, let $P_x: \R^{2d}\to \R^d$ is the map defined by
\begin{align*}
    (z,r)&\mapsto z+Xr^T=
    \begin{pmatrix}
    z_1\\
    z_2\\
    \vdots\\
    z_d
\end{pmatrix}+
\begin{pmatrix}
    x_1&0&\cdots &0\\
    0 &x_2&\cdots &0\\
    \vdots&\vdots &\ddots&\vdots\\
    0&0&\cdots&x_d.
\end{pmatrix}\cdot
\begin{pmatrix}
    r_1\\
    r_2\\
    \vdots\\
    r_d
\end{pmatrix}.
\end{align*}
Then we can write
\begin{align*}
    \tilde{\Gamma}(E)=\bigcup_{x\in E}P_x(\Tilde{\Gamma}).
\end{align*}
For each $x$, we define $T_x: \R^d\to \R^d$ by
\begin{align*}
    T_x(\xi)=X^T\cdot \xi^T=
    \begin{pmatrix}
    x_1&0&\cdots &0\\
    0 &x_2&\cdots &0\\
    \vdots&\vdots &\ddots&\vdots\\
    0&0&\cdots&x_d.
\end{pmatrix}\cdot
\begin{pmatrix}
    \xi_1\\
    \xi_2\\
    \vdots\\
    \xi_d
\end{pmatrix}=
\begin{pmatrix}
    x_1\xi_1\\
    x_2\xi_2\\
    \vdots\\
    x_d \xi_d
\end{pmatrix}.
\end{align*}
Let $\xi\in \R^d$, $|\xi|=1$, $p\in \R^d$, and $\delta>0$. We have
\begin{align*}
    |T_x(\xi)-p|\leq \delta \quad\Longrightarrow \quad |(\xi_1x_1-p_1,\cdots,\xi_dx_d-p_d)|\leq \delta.
\end{align*}
Since $|\xi|=1$, by pigeonholing, there is $1\leq k\leq d$ such that $|\xi_k|\geq\frac{1}{d^{1/2}}$.
Thus
\begin{align*}
    \{x\in E: |T_x(\xi)-p|\leq \delta\}\subset \bigcap\limits_{i=1}^d \{x\in E: |x_i\xi_i-p_i|\leq \delta\}\subset \{x\in E: |x_k\xi_k-p_k|\leq \delta\}.
\end{align*}
Put $H_{\xi_k}=\{x\in \R^d: x_k\xi_k=p_k\}$, then $H_{\xi_k}$ is a $(d-1)$-dimensional affine space in $\R^d$. Hence
\begin{align*}
    \{x\in E: |T_x(\xi)-p|\leq \delta\}\subset H_{\xi_k}(d^{1/2}\delta)\cap B(0,1).
\end{align*}
Since $H_{\xi_k}$ has dimension $d-1$, we can cover $H_\xi(d^{1/2}\delta)\cap B(0,1)$ by $\lesssim\delta^{-(d-1)}$ balls of radius $\delta$ with bounded overlaps. By Frostman condition, we have
\begin{align*}
    \mu\big(\{x\in E: |T_x(\xi)-p|\leq \delta\}\big)\lesssim \delta^{-(d-1)}\delta^{s_E}=\delta^{-d+1+s_E}.
\end{align*}
Thus \eqref{eq_condition_measure} holds with $\beta=-d+1+s_E$.

Apply Theorem \ref{thm_Oberlin_projection}, one can see that if
\begin{align*}
    2d-d+u<s_{\tilde{\Gamma}}+1+s_E-d \quad \Longleftrightarrow \quad s_E+s_{\tilde{\Gamma}}>2d-1+u,
\end{align*}
then for $\mu$-almost all $x\in \R^d$, 
\begin{align*}
    \dim_\mathcal{S}P_x(\tilde{\Gamma})\geq u.
\end{align*}
This finishes the proof of the theorem.
\begin{flushright}
$\Box$
\end{flushright}

\section{Union of sets by rigid motions over finite fields}\label{sec_5}
In this section, we will give proof of Theorem \ref{thm_FF_dim2} in dimension two. In higher dimensions, the proofs of Theorem \ref{Fur1} and Theorem \ref{Fur2} follow by using the same argument and the fact that the stabilizer of a non-zero element in $\mathbb{F}_q^d$ is about $|O(d-1)|$. We begin by reviewing some basic background on Fourier analysis over finite fields.

Let $\F_q^d$ be the $d-$dimensional vector space over a finite field $\F_q$ with $q$ elements, where $q$ is an odd prime power. The Fourier transform of a complex-valued function $f$ on $\F_q^d$ with respect to a nontrivial principal additive character $\chi$ on $\F_q$ is given by
\begin{align*}
    \widehat{f}(m)=q^{-d}\sum_{x\in \F_q^d}\chi(-x\cdot m)f(x),
\end{align*}
and the Fourier inversion formula takes the form
\begin{align*}
    f(x)=\sum_{m\in \F_q^d}\chi(x\cdot m)\widehat{f}(m).
\end{align*}
We also have the Plancherel theorem
\begin{align*}
    \sum_{m\in \F_d^d}|\widehat{f}(m)|^2=\frac{1}{q^d}\sum_{x\in \F_q^d}|f(x)|^2,
\end{align*}
which can be proved easily by using the following orthogonal property of the canonical additive character
\begin{align*}
    \sum_{m\in \F_q^d}\chi(m\cdot x)=
    \begin{cases}
        0& \text{ if } x\neq (0,\dots,0),\\
        q^d&\text{ if } x= (0,\dots,0).
    \end{cases}
\end{align*}
We will identify the set $E\subset \F_q^d$ with the characteristic function on the set $E$, and we denote by $|E|$ the cardinality of the set $E\subset \F_q^d$.

For $E\subset \mathbb{F}_q^d$, define 
\[M^*(E)=\max_{j\ne 0} \sum_{m\in S_j}|\widehat{E}(m)|^2, ~\mbox{and}~M(E)=\max_{j\in \mathbb{F}_q} \sum_{m\in S_j}|\widehat{E}(m)|^2.\]
By Plancherel, it is trivial that 
\[M(E), M^*(E)\le \frac{|E|}{q^d}.\]
We recall the following improvements from \cite{Chapmanetal12} and \cite{KohSun15}.
\begin{theorem}\label{m(A)}
Let $E\subset \mathbb{F}_q^d$. We have 
\begin{enumerate}
    \item If $d=2$, then $M^*(E)\ll q^{-3}|E|^{3/2}$. 
    \item If $d\ge 4$ even, then $M^*(E)\ll \min \left\lbrace \frac{|E|}{q^{d}},~ \frac{|E|}{q^{d+1}}+\frac{|E|^2}{q^{\frac{3d+1}{2}}}\right\rbrace$.
     \item If $d\ge 3$ odd, then $M(E)\ll  \min \left\lbrace \frac{|E|}{q^{d}},~ \frac{|E|}{q^{d+1}}+\frac{|E|^2}{q^{\frac{3d+1}{2}}}\right\rbrace$.
\end{enumerate}
\end{theorem}
In some specific dimensions, the same restriction estimate holds for the sphere of radius zero. A detailed proof can be found in \cite{Iosevichetal21}.
\begin{theorem}\label{zero-sphere}
    Let $E\subset \mathbb{F}_q^d$. Assume $d\equiv 2\mod{4}$ and $q\equiv 3\mod{4}$, then we have 
    \[\sum_{m\in S_0}|\widehat{E}(m)|^2\ll  \frac{|E|}{q^{d+1}}+\frac{|E|^2}{q^{\frac{3d+2}{2}}}.\]
\end{theorem}
Now we are ready to prove Theorem \ref{thm_FF_dim2}.

\subsection{Proof of Theorem \ref{thm_FF_dim2}}
Let $\lambda_\Theta$ be a function defined by 
\[\sum_{y\in \mathbb{F}_q^2}f(y)\lambda_\Theta(y):=\sum_{x\in E}\sum_{(g, z)\in \Theta}f(gx+z)\Theta(g, z),\]
for all functions $f\colon \mathbb{F}_q^2\to \mathbb{C}$. 

Let $f$ be the characteristic function of the set $\Theta(E)$. Then we have 
\[\sum_{x}f(x)\lambda_\Theta(x)=|E||\Theta|.\]
By the Cauchy-Schwarz inequality, one has 
\[|\Theta(E)|\ge \frac{|E|^2|\Theta|^2}{\sum_{x}\lambda_\Theta(x)^2}.\]
In the next step, we are going to bound $\sum_{x}\lambda_\Theta(x)^2$ from above. It suffices to prove that
\[\sum_{x}\lambda_\Theta(x)^2\ll \frac{|\Theta|^2|E|^2}{q^2}+q|E|^{3/2}|\Theta|.\]
We have 
\begin{align*}
    \widehat{\lambda_\Theta}(m)&=\frac{1}{q^2}\sum_{g, z, x}\chi(-m\cdot(g x+z))E(x)\Theta(g, z)=\sum_{g, z}\widehat{E}(g^{-1}m)\chi(-m\cdot z)\Theta(g, z)\\
    &=q^2\sum_{g}\widehat{E}(g^{-1}m)\left(\frac{1}{q^2}\sum_{z}\chi(-m\cdot z)\Theta(g, z)\right)=q^2\sum_{g}\widehat{E}(g^{-1}m)f_{g}(m),
\end{align*}
where 
\[f_{g}(m)= \frac{1}{q^2}\sum_{z}\chi(-m\cdot z)\Theta(g, z).\]

For $m=(0, 0)$, a direct computation shows that $\widehat{\lambda_{\Theta}}(0,0)=\frac{|E||\Theta|}{q^2}$. 

For $m\ne (0, 0)$, we can apply Cauchy-Schwarz to get
\begin{align*}
    &|\widehat{\lambda_\Theta}(m)|^2\le q^4\sum_{g}|\widehat{E}(g m)|^2\sum_{g}|f_g(m)|^2.
\end{align*}
Note that $||m||\ne 0$ when $m\ne (0, 0)$ since $q\equiv 3\mod 4$. In the plane $\mathbb{F}_q^2$, the stabilizer of a non-zero element is of size at most $2$. Thus, given $m$ with $||m||=j\ne 0$, by Theorem \ref{m(A)}, we have 
\[\sum_{g}|\widehat{E}(g m)|^2\ll M^*(E)\ll |E|^{3/2}q^{-3}.\]
So, 
\[|\widehat{\lambda_\Theta}(m)|^2\ll q^4\cdot \frac{|E|^{3/2}}{q^3}\cdot \sum_{g}|f_{g}(m)|^2,\]
which is bounded further by 
\[\frac{|E|^{3/2}}{q^3}\sum_{\theta}\sum_{u, v}\chi((u-v)\cdot m)\Theta(g, u)\Theta(g, v).\]
Taking the sum over all $m$ and use the orthogonality of $\chi$, one has 
\[\sum_{m}|\widehat{\lambda_\Theta}(m)|^2\ll \frac{|\Theta|^2|E|^2}{q^4}+\frac{|E|^{3/2}|\Theta|}{q}.\]
Using the fact that 
\[\sum_{x}\lambda_\Theta(x)^2=q^2\sum_{m}|\lambda_\Theta(m)|^2,\]
we obtain 
\[\sum_{x}\lambda_\Theta(x)^2\ll \frac{|\Theta|^2|E|^2}{q^2}+q|E|^{3/2}|\Theta|.\]
This completes the proof of the theorem.
\begin{flushright}
$\Box$
\end{flushright}

\section{Examples}\label{Section_examples}
In this section, we will give some examples regarding our results for packing problems using affine transformations.

In the first example, we want to emphasize that our results recover the known results for spheres in a special case.
\begin{example}
Let $E\subset \R^d$ be the unit sphere $S^{d-1}$ in $\R^d$, $d\geq 2$. Let $Z\subset \R^d$ be a set of translations, and put $\Theta=O(d)\times Z$. Due to Theorem \ref{thm_main_rigid}, since $S^{d-1}$ is a Salem set, if
        \begin{align*}
            \dim_\mathcal{H}Z >1
        \end{align*}
        then $\mathcal{L}^d(\Theta(S^{d-1}))>0$. In other words, the union $\bigcup_{z\in Z}(z+S^{d-1})$ of spheres with radius $1$, whose centers at $z\in Z$, has positive Lebesgue measure if $\dim_\mathcal{H}Z>1$. Similarly, if $\dim_\mathcal{H}Z >u$, for $0<u<1$, then we find that
        \begin{align*}
            \dim_\mathcal{H} \bigcup_{z\in Z}(z+S^{d-1}) \geq d-1+u.
        \end{align*}
        Thus our results recover the results by Wolff \cite{Wolff97}, \cite{Wolff00} and Oberlin \cite{Oberlin06} in the case spheres with a fixed given radius. 
\end{example}

Next, we will give some examples to illustrate the best dimensional thresholds that one can expect for packing problems using dilations and translations; rigid motions; and similarity transformations. Note that the Fourier dimension estimates in Theorems \ref{thm_main_dilation_translation}, \ref{thm_main_rigid}, and \ref{thm_main_similarity} must be sharp since the consequences for the measure estimates are sharp.

\begin{example}\label{example_2}
    \begin{itemize}
        \item[]
        \item[(i)] The bounds in Theorems \ref{thm_main_dilation_translation} $(ii)$ and \ref{thm_main_dilation_translation_general} $(ii)$ are sharp. By Theorem 1.3 in \cite{Korner08}, there exists a closed set $E\subset \R^d$ such that $\dim_\mathcal{H}E=d$, and $\mathcal{L}^d(E+E)=0$. Thus we can choose $T=\{1\}$, $Z=E$. Then one has 
    \begin{align*}
        \dim_\mathcal{H}E+\dim_\mathcal{H}Z+\dim_\mathcal{H}T=2d,
    \end{align*}
    while $\mathcal{L}^d(TE+Z)=0$. 
    \item[(ii)]  The bound in Theorem \ref{thm_main_multi_dilation_translation_general} $(ii)$ is also sharp. Let $d\geq 2$. Let $E=[0,1]^{d-1}\times \{0\}\subset \R^d$, $T=[0,1]^d$, and $A\subset [0,1]$ such that $\dim_{\mathcal{H}}A=1$, while $\mathcal{L}^1(A)=0$. Put $Z=[0,1]^{d-1}\times A$, and $\tilde{\Gamma}=Z\times T$. Then one can see that
    \begin{align*}
        \dim_{\mathcal{H}}E+\dim_{\mathcal{H}}\tilde{\Gamma}=3d-1,
    \end{align*}
    while $\mathcal{L}^d(\tilde{\Gamma}(E))=\mathcal{L}^d(Z)=0$.
    \end{itemize}   
\end{example}

\begin{example}\label{example_3}
     Let $d\geq 2$, $E=\R^{d-1}\times \{0\}$, and $G=\{g\in O(d): g(e_d)=e_d\}\cong O(d-1)$, the stabilizer of $e_d$ in $O(d)$. Let $A\subset [0,1]$ be a compact set of dimension $\dim_\mathcal{H}A=1$, such that $\mathcal{L}^1(A)=0$, and take $Z=[0,1]^{d-1}\times A$. Choose $\Theta=G\times Z\subset E(d)$, then we have
    \begin{align*}
        \dim_\mathcal{SF}E+\dim_\mathcal{H}\Theta
        =d-1+\frac{(d-1)(d-2)}{2}+d
        = \frac{d^2+d}{2},
    \end{align*}
    while $\mathcal{L}^d(\Theta(E))=0$. This illustrates that the bound in Theorem \ref{thm_main_rigid} $(ii)$ is sharp.

    To see that $\dim_\mathcal{SF}E=d-1$ let $\mu$ be any compactly supported measure on $E$ with $C^{\infty}$ density. Then $\widehat{\mu}(x,t) = \widehat{\mu}(x,0)$ for all $x\in\R^{d-1},t\in\R$. Let $0<\sigma < d-1$. Splitting the integration over the sphere $\{(x,t):|x|^2+t^2 = r^2\}$ to $|x|\leq r^{(d-1-\sigma)/(d-1)}$ and $|x|> r^{(d-1-\sigma)/(d-1)}$, we have for sufficiently large $N$ and any $r>1$, 
$$\int_{\{y\in\R^d:|y|=r\}}|\widehat{\mu}(y)|^2dy \lesssim 
r^{d-1-\sigma} + r^{-N} \leq 2r^{d-1-\sigma}.$$
\end{example}
\begin{example}\label{example_4}
    Let $d\geq 2$, by \cite[Theorem 1.3]{Korner08}, there exists a closed set $A\subset [0,1]$ such that $\dim_\mathcal{H}A=1$, and $\mathcal{L}^1(A+A)=0$. Thus we can choose $G$ as the stabilizer of $e_d$ in $O(d)$, $E=Z=[0,1]^{d-1}\times A\subset \R^d$, and put $\Theta=G\times Z$. Then one has
    \begin{align*}
        \frac{d-1}{d}\dim_{\mathcal{H}}E+\dim_{\mathcal{H}}\Theta
        =d-1+\frac{(d-1)(d-2)}{2}+d
        =\frac{d^2+d}{2},
    \end{align*}
    while $\mathcal{L}^d(\Theta(E))=0$. This implies that the dimensional threshold in Corollary \ref{thm_main_rigidcor} $(i)$ is sharp.
\end{example}

\begin{example}\label{example_5}
     Let $d\geq 2$, $E$, $G$, and $Z$ as in Example \ref{example_3}. Let $T=[1,2]$, and put $\Omega=G\times Z\times T$, then 
    \begin{align*}
    \dim_\mathcal{H}E+\dim_\mathcal{H}\Omega=\frac{d^2+d+2}{2},
    \end{align*}
    and $\mathcal{L}^d(\Theta(E))=0$. This shows that the bound in Theorem \ref{thm_main_similarity} $(ii)$ is sharp.
\end{example}

\section*{Acknowledgements}
The authors would like to thank the Vietnam Institute for Advanced Study in Mathematics (VIASM) for the hospitality and the excellent working conditions, where part of this work was done.

\bibliographystyle{siam}
\bibliography{ref} 

\begin{thebibliography}{10}

\bibitem{Besicovitch20}
{\sc A.~S. Besicovitch}, {\em {Sur deux questions d'int{\'e}grabilit{\'e} des fonctions}}, J. Soc. Phys. Math, Perm, 2 (1919), pp.~105--123.

\bibitem{Besicovitch28}
\leavevmode\vrule height 2pt depth -1.6pt width 23pt, {\em On {K}akeya's problem and a similar one}, Math. Z., 27 (1928), pp.~312--320.

\bibitem{BesicovitchRado68}
{\sc A.~S. Besicovitch and R.~Rado}, {\em A plane set of measure zero containing circumferences of every radius}, J. Lond. Math. Soc., 43 (1968), pp.~717--719.

\bibitem{Bourgain85}
{\sc J.~Bourgain}, {\em Estimations de certaines fonctions maximales}, C.R. Acad. Sci. Paris S\'er. I Math., 301 (1985), pp.~499--512.

\bibitem{Bourgain86}
\leavevmode\vrule height 2pt depth -1.6pt width 23pt, {\em Averages in the plane over convex curves and maximal operators}, J. Anal. Math., 47 (1986), pp.~69--85.

\bibitem{Bourgain10}
\leavevmode\vrule height 2pt depth -1.6pt width 23pt, {\em The discretized sum-product and projection theorems}, J. Anal. Math., 112 (2010), pp.~193--236.

\bibitem{Changetal18}
{\sc A.~Chang, M.~{Cs\"ornyei}, K.~{H\'era}, and T.~Keleti}, {\em Small unions of affine subspaces and skeletons via baire category}, Adv. Math., 328 (2018), pp.~801--821.

\bibitem{Chapmanetal12}
{\sc J.~Chapman, M.~B. Erdo{\u{g}}an, D.~Hart, A.~Iosevich, and D.~Koh}, {\em Pinned distance sets, {$k$}-simplices, {W}olff’s exponent in finite fields and sum-product estimates}, Math. Z., 271 (2012), pp.~63--93.

\bibitem{Davies72}
{\sc R.~O. Davies}, {\em Another thin set of circles}, J. Lond. Math. Soc., 5 (1972), pp.~191--192.

\bibitem{Duetal21}
{\sc X.~Du, L.~Guth, Y.~Ou, H.~Wang, B.~Wilson, and R.~Zhang}, {\em Weighted restriction estimates and application to {F}alconer distance set problem}, Amer. J. Math., 142 (2021), pp.~175--211.

\bibitem{DuZhang19}
{\sc X.~Du and R.~Zhang}, {\em Sharp {$L^2$} estimates of the {S}ch{r\"{o}}dinger maximal function in higher dimensions}, Ann. of Math., 189 (2019), pp.~837--861.

\bibitem{Erdoganetal13}
{\sc B.~Erdo{\u{g}}an, D.~Hart, and A.~Iosevich}, {\em Multiparameter projection theorems with applications to sums-products and finite point configurations in the euclidean setting}, in Recent Advances in Harmonic Analysis and Applications, New York, NY, 2013, Springer New York, pp.~93--103.

\bibitem{Falconer82}
{\sc K.~J. Falconer}, {\em Hausdorff dimension and the exceptional set of projections}, Mathematika, 29 (1982), pp.~109--115.

\bibitem{Falconer85book}
\leavevmode\vrule height 2pt depth -1.6pt width 23pt, {\em The Geometry of Fractal Sets}, Cambridge Tracts in Mathematics, Cambridge University Press, 1985.

\bibitem{FalconerMattila16}
{\sc K.~J. Falconer and P.~Mattila}, {\em Strong {M}arstrand theorems and dimensions of sets formed by subsets of hyperplanes}, J. Fractal Geom., 3 (2016), pp.~319--329.

\bibitem{Gan23}
{\sc S.~Gan}, {\em Hausdorff dimension of unions of {$k$-}planes}, arxiv:2305.14544,  (2023).

\bibitem{Hametal23}
{\sc S.~Ham, H.~Ko, S.~Lee, and S.~Oh}, {\em Remarks on dimension of unions of curves}, Nonlinear Anal., 229 (2023), pp.~113--207.

\bibitem{HambrookTaylor21}
{\sc K.~Hambrook and K.~Taylor}, {\em Measure and dimension of sums and products}, Proc. Amer. Math. Soc., 149 (2021), pp.~3765--3780.

\bibitem{Hera19}
{\sc K.~H{\'e}ra}, {\em Hausdorff dimension of {F}urstenberg-type sets associated to families of affine subspaces}, Ann. Acad. Sci. Fenn. Math., 44 (2019), pp.~903--923.

\bibitem{Heraetal19}
{\sc K.~{H\'era}, T.~Keleti, and A.~{M\'ath\'e}}, {\em Hausdorff dimension of unions of affine subspaces and of {F}urstenberg-type sets}, J. Fractal Geom., 6 (2019), pp.~263--284.

\bibitem{Iosevichetal21}
{\sc A.~Iosevich, D.~Koh, S.~Lee, T.~Pham, and C.-Y. Shen}, {\em On restriction estimates for the zero radius sphere over finite fields}, Canad. J. Math., 73 (2021), p.~769–786.

\bibitem{Kaenmaki17}
{\sc A.~K\"aenm\"aki, T.~Orponen, and L.~Venieri}, {\em A {M}arstrand-type restricted projection theorem in {$\mathbb{R}^d$}}, arXiv:1708.04859,  (2017).

\bibitem{Keleti17}
{\sc T.~Keleti}, {\em Small union with large set of centers}, in Recent Developments in Fractals and Related Fields, J.~Barral and S.~Seuret, eds., Cham, 2017, Springer International Publishing, pp.~189--206.

\bibitem{Keletietal18}
{\sc T.~Keleti, D.~T. Nagy, and P.~Shmerkin}, {\em Squares and their centers}, J. Anal. Math., 134 (2018), pp.~643--669.

\bibitem{Kinney68}
{\sc J.~R. Kinney}, {\em A thin set of circles}, Amer. Math. Monthly, 75 (1968), pp.~1077--1081.

\bibitem{KohSun15}
{\sc D.~Koh and H.-S. Sun}, {\em Distance sets of two subsets of vector spaces over finite fields}, Proc. Amer. Math. Soc., 143 (2015), pp.~1679--1692.

\bibitem{Korner08}
{\sc T.~W. K\"orner}, {\em Hausdorff dimension of sums of sets with themselves}, Studia Math., 188 (2008), pp.~287--295.

\bibitem{Marstrand87}
{\sc J.~M. Marstrand}, {\em {P}acking circles in the plane}, Proc. London. Math. Soc., s3-55 (1987), pp.~37--58.

\bibitem{Mattila87}
{\sc P.~Mattila}, {\em Spherical averages of {F}ourier transforms of measures with finite energy; dimensions of intersections and distance sets}, Mathematika, 34 (1987), p.~207–228.

\bibitem{Mattila95}
\leavevmode\vrule height 2pt depth -1.6pt width 23pt, {\em Geometry of Sets and Measures in Euclidean Spaces: {F}ractals and {R}ectifiability}, Cambridge Studies in Advanced Mathematics, Cambridge University Press, 1995.

\bibitem{Mattila15}
\leavevmode\vrule height 2pt depth -1.6pt width 23pt, {\em {F}ourier {A}nalysis and {H}ausdorff {D}imension}, Cambridge Studies in Advanced Mathematics, Cambridge University Press, 2015.

\bibitem{Mattila17}
\leavevmode\vrule height 2pt depth -1.6pt width 23pt, {\em Exceptional set estimates for the {H}ausdorff dimension of intersections}, Ann. Acad. Sci. Fenn. Math., 42 (2017), pp.~611--620.

\bibitem{Mattila22}
\leavevmode\vrule height 2pt depth -1.6pt width 23pt, {\em {Hausdorff dimension and projections related to intersections}}, Publ. Mat., 66 (2022), pp.~305 -- 323.

\bibitem{Mitsis99}
{\sc T.~Mitsis}, {\em On a problem related to sphere and circle packing}, J. Lond. Math. Soc., 60 (1999), pp.~501--516.

\bibitem{Oberlin06}
{\sc D.~M. Oberlin}, {\em Packing spheres and fractal {S}trichartz estimates in {$\R^d$} for {$d\geq 3$}}, Proc. Amer. Math. Soc., 134 (2006), pp.~3201--3209.

\bibitem{Oberlin07}
\leavevmode\vrule height 2pt depth -1.6pt width 23pt, {\em {Unions of hyperplanes, unions of spheres, and some related estimates}}, Illinois J. Math., 51 (2007), pp.~1265 -- 1274.

\bibitem{Oberlin14}
\leavevmode\vrule height 2pt depth -1.6pt width 23pt, {\em Exceptional sets of projections, unions of {$k-$}planes and associated transforms}, Isr. J. Math., 202 (2014), pp.~331 -- 342.

\bibitem{ROberlin16}
{\sc R.~Oberlin}, {\em Unions of lines in {$\F^n$}}, Mathematika, 62 (2016), p.~738–752.

\bibitem{Pham23}
{\sc T.~Pham}, {\em Triangles with one fixed side–length, a {F}urstenberg-type problem, and incidences in finite vector spaces}, Forum Math.,  (2024).

\bibitem{RenWang23}
{\sc K.~Ren and H.~Wang}, {\em Furstenberg sets estimate in the plane}, arXiv:2308.08819,  (2023).

\bibitem{SimonTaylor20}
{\sc K.~Simon and K.~Taylor}, {\em Interior of sums of planar sets and curves}, Math. Proc. Camb. Philos. Soc., 168 (2020), p.~119–148.

\bibitem{SimonTaylor22}
\leavevmode\vrule height 2pt depth -1.6pt width 23pt, {\em Dimension and measure of sums of planar sets and curves}, Mathematika, 68 (2022), pp.~1364--1392.

\bibitem{Stein76}
{\sc E.~M. Stein}, {\em Maximal functions: spherical means}, Proc. Nat. Acad. Sci. U.S.A., 73 (1976), pp.~2174--2175.

\bibitem{Talagrand80}
{\sc M.~Talagrand}, {\em Sur la mesure de la projection d'un compact set certaines families de cercles}, Bull. Sci. Math., 104 (1980), pp.~225--231.

\bibitem{Thorton17}
{\sc R.~Thornton}, {\em Cubes and their centers}, Acta Math. Hungar., 152 (2017), pp.~291--313.

\bibitem{Wisewell04}
{\sc L.~Wisewell}, {\em Families of surfaces lying in a null set}, Mathematika, 51 (2004), p.~155–162.

\bibitem{Wolff97}
{\sc T.~Wolff}, {\em A {K}akeya-type problem for circles}, Amer. J. Math., 119 (1997), pp.~985--1026.

\bibitem{Wolff99}
\leavevmode\vrule height 2pt depth -1.6pt width 23pt, {\em {Decay of circular means of Fourier transforms of measures}}, Int. Math. Res. Not., 1999 (1999), pp.~547--567.

\bibitem{Wolff00}
\leavevmode\vrule height 2pt depth -1.6pt width 23pt, {\em {Local smoothing type estimate on {$L^p$} for large {$p$}}}, Geom. Funct. Anal., 10 (2000), pp.~1237--1288.

\bibitem{Zahl23}
{\sc J.~Zahl}, {\em Unions of lines in {$\mathbb{R}^n$}}, Mathematika, 69 (2023), pp.~473--481.

\end{thebibliography}

\end{document}